\documentclass{amsart}
\usepackage[utf8]{inputenc}
\usepackage{amsfonts}
\usepackage{amsmath}
\numberwithin{equation}{section}
\usepackage{amssymb}
\usepackage{float}
\usepackage{tikz-cd}
\usepackage{amsthm}
\usepackage{changes}
\usepackage{enumerate}
\usepackage{tikz}
\usepackage{graphicx}
\usepackage{hyperref}
\usepackage[capitalise]{cleveref}
\usepackage{slashed}  % For letters with a slash
\usepackage{upgreek}
\usepackage[top=3cm,bottom=3cm,left=2.5cm,right=2.5cm,headsep=10pt,letterpaper]{geometry} % Page margins
\usepackage{fancyhdr}
\usepackage{mathrsfs}
\usepackage{dsfont}

\usepackage[
    backend=biber,
    style=alphabetic,
    sorting=nyt
]{biblatex}
\numberwithin{equation}{section}
\theoremstyle{plain}
\newtheorem*{conjecture}{Conjecture}

\newtheorem{theorem}{Theorem}[section]
\newtheorem{proposition}[theorem]{Proposition}
\newtheorem{lemma}[theorem]{Lemma}
\newtheorem{corollary}[theorem]{Corollary}

\theoremstyle{definition}
\newtheorem{definition}[theorem]{Definition}

\theoremstyle{remark}
\newtheorem{remark}[theorem]{Remark}

\theoremstyle{plain}

\newcommand{\R}{\mathbb{R}}

\newcommand{\grad}{\nabla}
\newcommand{\RR}{\mathbb{R}}

\DeclareMathOperator{\Ric}{Ric}

\newcommand{\CP}{\mathbb{CP}}

\newcommand{\Hess}{\operatorname{Hess}}
\newcommand{\II}{\mathrm{I\!I}}

\newcommand{\Id}{\operatorname{Id}}

\DeclareMathOperator{\tr}{tr}

\title{Log-Concavity of First Dirichlet Eigenfunctions on \(\CP^2\)}

\author{Yusen Xia}
\address{Department of Mathematics, University of California, Santa Barbara, CA 93106, USA}
\email{yusen@ucsb.edu}
\thanks{Y. Xia is partially supported by NSF DMS 2403557.}

\begin{document}

\begin{abstract}
We study the log-concavity property of first Dirichlet eigenfunctions on domains
in $\CP^2$. For every smooth $1$-convex
domain $\Omega\subset\CP^2$, we prove the quantitative estimate
\[
\nabla^2(-\log u)
>
\max\left\{\psi(s),\frac85\right\}g,
\text{ where }
\psi(|\grad f|^2)
=
\frac{s}{\sqrt{1+s}}-\log(1+s),
\]
for its first Dirichlet eigenfunction $u$. In particular, $u$ is
strictly log-concave.

As consequences, we obtain a uniform convexity estimate for the regular
level sets of $u$ and the fundamental gap bound
$\lambda_2-\lambda_1>46/5$.
\end{abstract}

\maketitle

\tableofcontents

\section{Introduction and main results}

Let $\Omega$ be a bounded, smooth, connected domain in a Riemannian
manifold $(M,g)$, and let $u>0$ be its first Dirichlet eigenfunction, i.e.
\[
-\Delta u=\lambda_1u
\quad\text{in }\Omega,
\qquad
u=0
\quad\text{on }\partial\Omega.
\]
A fundamental question in spectral geometry is whether $u$ is log-concave, namely whether
\[
\Hess(\log u)\le0.
\]

In this paper we study this problem on complex projective space.
The corresponding question on $\CP^n$ was raised by
Dai--Seto--Wei \cite{DaiSetoWei2019}, where the following conjecture
was attributed to Zhiqin Lu.

\begin{conjecture}[Log-concavity conjecture in $\CP^n$]
Let $\Omega\subset\CP^n$ be a convex domain. Then the first Dirichlet
eigenfunction of $\Omega$ is log-concave.
\end{conjecture}

We consider the conjecture in complex dimension two, with $\CP^2$
equipped with the Fubini--Study metric normalized so that the holomorphic
sectional curvature is $4$. In contrast with earlier results in the space forms
$\RR^n$, $\mathbb S^n$, and $\mathbb H^n$ (see \cref{subsec:abridged history} for a brief history), the sectional curvature of
$\CP^2$ is not constant but ranges from $1$ to $4$. Geometrically, the
curvature operator distinguishes the direction $JX$ from directions
orthogonal to the complex line spanned by $X$ and $JX$. This anisotropy
creates the main obstruction to the standard Hessian maximum-principle
arguments.

We obtain an affirmative result toward the conjecture under a stronger convexity assumption of the domain. To state the result, we use the following quantitative notion of boundary convexity (cf. \cite[Definition~2]{WeiXiao2025}).

\begin{definition}
We say that a smooth domain $\Omega$ is $\kappa$-convex if, with respect
to the inward unit normal,
\[
\II_{\partial\Omega}\ge\kappa g|_{T_{\partial\Omega}}
\] where we follow the convention $\II_{\partial \Omega}(V,W)
=-\langle\nabla_V\nu,W\rangle.$
\end{definition}
In hyperbolic space, $1$-convexity coincides with horo-convexity for smooth domains.

We introduce the following barrier function to quantify convexity of $-\log u$:
\[
\psi(s)
:=
\frac{s}{\sqrt{1+s}}-\log(1+s),
\quad \forall  s\ge0.
\]

\begin{theorem}\label{thm: main}
Let $\Omega\subset\CP^2$ be a connected $1$-convex domain with smooth
boundary, and let $u>0$ be its first Dirichlet eigenfunction. Denote $f:=-\log u.$
Then we have a gradient-dependent bound of $\grad^2 f$
\begin{equation}\label{eq:main-gradient-bound}
\nabla^2 f>\psi(|\grad f|^2)g,
\end{equation}
and a uniform bound
\begin{equation}\label{eq:main-uniform-bound}
\nabla^2 f>\frac85 g.
\end{equation}
Equivalently,
\[
\nabla^2 f>
\max\left\{\psi(|\grad f|^2),\frac85\right\}g.
\]
In particular, $u$ is strictly log-concave.
\end{theorem}

The two lower bounds in \cref{thm: main} are complementary. The uniform
bound $8/5$ is strongest where $|\nabla f|$ is small, including at critical
points of $u$, while the gradient-dependent bound $\psi(s)$ grows like
$\sqrt{s}$ as $s\to\infty$ and is therefore much stronger near the boundary.

The quantitative Hessian estimates have both geometric and spectral
consequences.

\begin{corollary}[Uniform convexity of level sets]
\label{cor:level-set-convexity}
Let $\Omega\subset\CP^2$ be a connected $1$-convex domain with smooth
boundary, and let $u>0$ be its first Dirichlet eigenfunction. Then every
regular level set $\Sigma_c=\{u=c\}$ satisfies
\[
    \II_{\Sigma_c}
    >
    \frac{8}{25}g|_{T\Sigma_c}.
\]
In particular, every regular superlevel set of $u$ is uniformly convex.
\end{corollary}

\begin{corollary}[Fundamental gap estimate]
\label{cor:fundamental-gap}
Let $\Omega\subset\CP^2$ be a connected $1$-convex domain with smooth
boundary. Then its first two Dirichlet eigenvalues satisfy
\[
    \lambda_2-\lambda_1>\frac{46}{5}.
\]
\end{corollary}

To the best of our knowledge, \cref{thm: main} provides the first affirmative
result toward the log-concavity conjecture on complex projective space
beyond the elementary case of geodesic balls. It yields a
strict quantitative improvement of ordinary log-concavity under the boundary condition of $1$-convexity. The gradient-dependent estimate obtained in this way also provides the large-gradient control needed to derive the uniform Hessian lower bound. As consequences, we obtain uniform convexity of the regular level sets and a universal lower bound for the Dirichlet fundamental gap within this class of domains. The proof adapts the first-contact framework of \cite{KNTW} to the anisotropic curvature of $\CP^2$, using a carefully chosen scalar barrier in the modified Hessian tensor together with a determinant estimate at a hypothetical interior contact point.

We next recall the development of the log-concavity problem and its
connection with fundamental-gap estimates.
\subsection{An abridged history of the log-concavity problem}\label{subsec:abridged history}
The study of log-concavity of first eigenfunctions dates back to Brascamp--Lieb \cite{brascamp1976extensions}, where they proved that the first Dirichlet
eigenfunction of a convex domain in $\mathbb{R}^n$ is log-concave.
Related strict-concavity results and further developments were obtained in
\cite{CaffarelliFriedman1985,SingerWongYauYau1985}. A major advance
was made by Andrews--Clutterbuck \cite{AndrewsClutterbuck2011}, who
established a sharp modulus of concavity for the logarithm of the first
eigenfunction on convex Euclidean domains; this quantitative refinement
of log-concavity was the key ingredient in their proof of the fundamental
gap conjecture. For convex domains
on the sphere, Lee--Wang \cite{LeeWang1987estimate} proved log-concavity by a
continuity argument, and subsequent work obtained stronger quantitative
estimates; see
\cite{SetoWangWei2019,HeWeiZhang2020,DaiSetoWei2021}.
For surfaces of variable positive curvature, log-concavity estimates were
obtained by Khan--Nguyen--Tuerkoen--Wei
\cite{KNTW}. 
% More recently, Khan--Saha--Tuerkoen \cite{KST2025} obtained a priori
% quantitative Hessian estimates for first Dirichlet eigenfunctions under an
% assumed log-concavity condition for domains on a general Riemannian manifold. 
In a different direction, Aryan--Law \cite{AryanLaw2025} proved a
general concavity principle for elliptic and parabolic equations on locally
symmetric spaces with nonnegative sectional curvature, including complex
projective spaces, although the log-concavity problem of the first Dirichlet eigenfunction
lies \emph{outside} the scope of their result.

The connection between log-concavity and the Dirichlet fundamental gap was
already exploited by Singer--Wong--Yau--Yau \cite{SingerWongYauYau1985}, who observed that
log-concavity of the first eigenfunction can be used to compare the
Dirichlet gap with a Neumann eigenvalue. Yau \cite{Yau2003Estimate} further obtained improved gap estimates from quantitative control of the log-concavity of
the first eigenfunction. The sharp modulus of concavity of Andrews--Clutterbuck
\cite{AndrewsClutterbuck2011} ultimately yielded the optimal Euclidean
estimate
\[
\lambda_2-\lambda_1\ge\frac{3\pi^2}{D^2}.
\]
Analogous sharp gap estimates on the sphere were later
proved in \cite{SetoWangWei2019,HeWeiZhang2020,DaiSetoWei2021}.

The geometry of the ambient manifold plays an essential role. In
negative curvature, ordinary convexity is not sufficient: Shih \cite{Shih1989} constructed a convex domain in hyperbolic space
whose first eigenfunction is not log-concave. Moreover,
Bourni--Clutterbuck--Nguyen--Stancu--Wei--Wheeler
\cite{BCNSWW2022} showed that the normalized fundamental gap of convex
domains in $\mathbb H^n$ can be arbitrarily small, precluding any
uniform quantitative log-concavity estimate of the Euclidean or
spherical type for general convex domains. On the other hand, stronger
boundary convexity can restore the phenomenon. Wei--Xiao
\cite{WeiXiao2025} proved \emph{super} log-concavity for
\emph{horo}-convex domains in $\mathbb H^n$ with a diameter restriction, and very recently
Dai--Ennis--Nguyen--Wei \cite{dai2026logconcavity} obtained
log-concavity and level-set horoconvexity for bounded horoconvex domains
in $\mathbb{H}^2$ without a diameter restriction.

\subsection{Proof strategy and organization}

Our proof follows the deformation and first-contact framework on general Riemannian manifolds
in \cite{KNTW}. We first prove \eqref{eq:main-gradient-bound}. By Huisken's convergence result for mean curvature flow of convex hypersurfaces \cite{Huisken1986}, we can construct a deformation from a sufficiently small and nearly round domain to the target domain, preserving $1$-convexity. We establish the desired Hessian inequality at the initial end of the deformation and in a uniform boundary collar. The problem is therefore reduced to excluding a hypothetical interior first contact point. To rule out any interior first contact, we consider the modified Hessian
\[
K
=
\nabla^2(-\log u)
-
\psi\bigl(|\nabla\log u|^2\bigr)g.
\]
The $1$-convexity assumption provides the boundary positivity of $K$, while
a Riccati comparison gives a uniform lower bound for the first Dirichlet
eigenvalue. At an interior null contact, the anisotropic curvature terms
of $\CP^2$ are controlled by a determinant estimate. After proving simple
nullity of $K$ at any contact point, the remaining finite-dimensional optimization reduces to three
extremal orientations of $\nabla\log u$ relative to the complex line
spanned by $X$ and $JX$.

We then use this estimate \eqref{eq:main-gradient-bound}, together with lower bound of $\lambda_1$ and a second determinant argument, to obtain the uniform bound \eqref{eq:main-uniform-bound}.
The two estimates are complementary and yield the uniform convexity of
regular level sets and the fundamental-gap estimate. Both corollaries are proved in
\cref{sec:uniform-hessian-bound}. The first follows by combining the two
Hessian lower bounds in \cref{thm: main}, while the second follows from the
uniform Hessian bound and the weighted Bochner formula for the ground-state
transform.

The paper is organized as follows. \cref{sec:prelim} collects the geometric
preliminaries and the first-contact principle. Section~\ref{sec:modified-hessian} introduces the modified Hessian
and establishes its boundary positivity. Section~\ref{sec:The first-contact calculation in CP2} contains
a lower bound estimate on $\lambda_1(\Omega)$ and the interior first-contact analysis, and
Section~\ref{sec:Proof of first main theorem} completes the deformation argument and the proof of the gradient-dependent bound \eqref{eq:main-gradient-bound} in the main theorem. \cref{sec:uniform-hessian-bound} proves the uniform Hessian lower bound and its two corollaries.
Appendix~A contains the auxiliary scalar inequalities and
finite-dimensional computations used in the proof of \cref{thm: main}.
\subsection*{Acknowledgments} The author would like to express his sincere gratitude to his doctoral advisor, Professor Guofang Wei, for suggesting this problem, for many insightful discussions,
and for her generous guidance, encouragement, and support throughout the
development of this work. The author also thanks Jia-Lin Hsu for helpful discussions and valuable comments on this work.

\subsection*{AI disclosure statement} The author used OpenAI's ChatGPT Pro as an assistive tool in brainstorming, mathematical development, and the preparation of this manuscript. In particular, ChatGPT assisted in identifying the exact choice of the barrier function $\psi$
used in the modified tensor, exploring auxiliary arguments, pre-checking calculations, and improving the exposition. All AI-assisted suggestions were independently checked, developed, and verified by the author, who takes full responsibility for the content of the paper.

%&& ALL Blue parts above has been used. Sep 9 7pm

%Nice thanks

%nice way to talk lol. I feel that we are AI agents constructing message boards.

\section{Preliminaries}\label{sec:prelim}

\subsection{Fubini-Study metric on the complex projective space $\CP^n$}

%Throughout the paper, we equip the complex projective space $(\CP^n,J,g)$ with the Fubini-Sutdy metric, which is of constant holomorphic sectional curvature $4$.
%With respect to the following convention for Riemannian curvature tensor,
%\begin{align*}
%    R(U,V)W&=\nabla_U\nabla_VW-\nabla_V\nabla_UW-\nabla_{[U,V]}W,
%\end{align*}
%we have 
%\[
%    R_X(Z):=R(Z,X)X=Z-\langle Z,X \rangle X+3\langle Z,JX \rangle JX.
%\]
%More generally, 

Throughout the paper,
we equip the complex projective space $\CP^n$ with the Fubini–Study metric \(g\), normalized so that the holomorphic sectional curvature is \(4\). Thus $\CP^n$ has real dimension $2n$, and sectional curvatures are pinched:
$
1\leq \sec \leq 4,
$
and 
$
\operatorname{Ric}= (2n+2)g.
$

We use the following convention for the Riemannian curvature tensor $R=R_{\mathbb{CP}^n}$:
$$
R(U,V)W
=
\nabla_U\nabla_VW-\nabla_V\nabla_UW-\nabla_{[U,V]}W.
$$
With this convention, the curvature tensor of $\CP^n$ can be written as follows:
$$
R(U,V)W = \langle V,W\rangle U-\langle U,W\rangle V +\langle JV,W\rangle JU-\langle JU,W\rangle JV
-2\langle JU,V\rangle JW.
$$

In particular, if \(X\) is a unit tangent vector and
$
R_X(Z):=R(Z,X)X
$ is the Jacobi operator,
then
\begin{equation}
R_X(Z)
=
Z-\langle Z,X\rangle X
+3\langle Z,JX\rangle JX.\label{FS curvature}
\end{equation}
Consequently, \(R_X\) has eigenvalue \(0\) in the \(X\)-direction, eigenvalue \(4\) in the \(JX\)-direction, and eigenvalue \(1\) on the orthogonal complement
$
\{X,JX\}^{\perp}.
$
Equivalently, for orthonormal vectors \(X,Z\), we have
\(
\sec(X,Z)
=
1+3\langle Z,JX\rangle^{2}.
\)
The distinction between the curvature-\(4\) direction \(JX\) and the $2n-2$ curvature-\(1\) directions orthogonal to \(\{X,JX\}\) is the main geometric feature of $\CP^n$ that enters the first-contact calculation below.

Moreover, the Fubini–Study metric is locally symmetric, so
$
\nabla R=0.
$ 

\subsection{First contact principle}
We recall the first-contact principle used in the continuation argument
for the modified Hessian introduced below.

Let \(\{\Omega_t\}_{0\le t\le1}\) be a smooth family of bounded domains, and suppose that there are smooth diffeomorphisms
$
F_t:\Omega_0\longrightarrow \Omega_t.
$
For each \(t\), let \(u_t>0\) denote the first Dirichlet eigenfunction of \(\Omega_t\), normalized by $\int_{\Omega_t} u_t^2=1.$
Since the first Dirichlet eigenvalue is simple, after pulling the eigenvalue problem back to the fixed domain \(\Omega_0\), standard elliptic perturbation theory implies that \(\lambda_t\) and \(u_t\) depend smoothly on \(t\), with the corresponding \(C^k\)-dependence on compact subsets of the interior. We will use this without further comment.

The following formulation applies to any symmetric two-tensor arising continuously along such a deformation.
\begin{lemma}[First contact principle]\label{lem: First Contact Principle}
Let \(\{\Omega_t\}_{0\le t\le1}\) be a smooth family of domains and \(A_t\) a $C^2$ family of symmetric two-tensors. Suppose
$A_0>0$
and that \(A_t>0\) in a uniform boundary collar $\{p\in \Omega_t| d(p,\partial\Omega_t)< \epsilon_0\}$ for every \(t\).

If positivity of $A_t$ fails at some time, then there exist a first time \(t_*\in(0,1]\), an interior point \(p\in\Omega_{t_*}\), and a unit vector \(X\in T_p\CP^2\) such that
$$
A_{t_*}\ge0,
\qquad
A_{t_*}(X,X)=0.
$$

In particular,
$
A_{t_*}(X,\cdot)=0.
$
If \(X\) is extended locally with \(\nabla X(p)=0\), then

$$
\nabla\bigl(A_{t_*}(X,X)\bigr)(p)=0,
\qquad
\Delta\bigl(A_{t_*}(X,X)\bigr)(p)\ge0.
$$

Moreover, if
$
\dim\ker A_{t_*}(p)=1,
$
then

$$
\Delta\det A_{t_*}(p)\ge0.
$$
    
\end{lemma}
\begin{proof}
By continuity, there is a first time $t_*\in(0,1]$ at which positivity is lost. At this time,
$
A_{t_*}\ge 0,
$
and there are a point $p\in\Omega_{t_*}$ and a unit vector $X\in T_p\mathbb{CP}^2$ such that $A_{t_*}(X,X)=0.$
The uniform boundary positivity implies that $p$ lies in the interior. Since $A_{t_*}$ is positive semi-definite,
$A_{t_*}(X,\cdot)=0.$
Extend $X$ locally so that $\nabla X(p)=0$. Then the scalar function $A_{t_*}(X,X)$ is nonnegative near $p$ and vanishes at $p$, hence has a local minimum there. Therefore
$$
\nabla\bigl(A_{t_*}(X,X)\bigr)(p)=0,
\qquad
\Delta\bigl(A_{t_*}(X,X)\bigr)(p)\ge 0.
$$

If $\ker A_{t_*}(p)=\mathbb{R}X$, then $\det A_{t_*}\ge0$ near $p$ and $\det A_{t_*}(p)=0$. Thus $\det A_{t_*}$ also has a local minimum at $p$, and therefore
$
\Delta\det A_{t_*}(p)\ge0.
$
\end{proof}
% In the proof of the main theorem, \(A_t\) will be the modified Hessian \(K_t\). The boundary estimate will rule out any possible first contact near \(\partial\Omega_t\), while the interior calculation will show that the final inequality above is incompatible with a null direction.

\section{The modified Hessian tensor and its boundary behavior}\label{sec:modified-hessian}

Let \(u>0\) be the first Dirichlet eigenfunction of $\Omega$,
\[
    -\Delta u=\lambda u,
    \qquad
    u|_{\partial\Omega}=0,
\]
and put $f:=-\log(u)$, $H:=\nabla^2 f$.
Then
\begin{equation}\label{eq:cp2-f-equation}
    \Delta f=\lambda+ |\nabla f|^2.
\end{equation}

We consider the following scalar barrier
\begin{equation}\label{eq:cp2-barrier-definition}
    \psi(s)
    :=
    \frac{s}{\sqrt{1+s}}
    -
    \log(1+s),
    \qquad
    s\geq0,
\end{equation}
and define the corresponding modified Hessian 
\begin{equation}\label{eq:cp2-K-definition}
    K
    :=
    H-\psi(|\nabla f|^2)g.
\end{equation}
Since \(\psi\geq0\), positivity of \(K\) is stronger than the desired
log-concavity:
\[
    K>0
    \quad\Longrightarrow\quad
    H=K+\psi(s)g>0
    \quad\Longrightarrow\quad
    \nabla^2\log u<0.
\]

\begin{lemma}[Elementary properties of the barrier]
\label{lem:cp2-barrier-properties}
Let \(\psi\) be defined by
\cref{eq:cp2-barrier-definition}. Then
\[
    \psi(0)=0,
    \qquad
    \psi(s)>0
    \quad
    \text{for }s>0.
\]
% Writing
% \[
%     r=\sqrt{1+s},
% \]
% one has
% \begin{align}
%     \psi
%     &=
%     r-\frac1r-2\log r,
%     \label{eq:psi-r}\\
%     \psi'
%     &=
%     \frac{(r-1)^2}{2r^3},
%     \label{eq:psi-prime}\\
%     \psi''
%     &=
%     \frac{(r-1)(3-r)}{4r^5}.
%     \label{eq:psi-second}
% \end{align}
% \begin{equation}\label{eq:psi-combination-one}
%     \psi'+2s\psi''
%     =
%     \frac{(r-1)^2(2r+3)}{2r^5}>0
%     \qquad
%     (s>0),
% \end{equation}
% and
\begin{equation}\label{eq:psi-combination-two}
    (\psi')^2+\psi\psi''\geq0.
\end{equation}
\begin{equation}\label{eq:psi-large-s}
    \psi(s)
    =
    \sqrt{s}-\log s+O(s^{-1/2})
    \qquad
    \text{as }s\to\infty.
\end{equation}
Moreover,
\begin{equation}\label{eq:psi-combination-one}
    \psi'+2s\psi''
    =
    \frac{(\sqrt{1+s}-1)^2(2\sqrt{1+s}+3)}{2(1+s)^{5/2}}>0
    \qquad
    (s>0),
\end{equation}
\end{lemma}

\begin{proof}
    Writing $r=\sqrt{1+s},$ one has
\begin{align}
    \psi
    &=
    r-\frac1r-2\log r,
    \label{eq:psi-r}\\
    \psi'
    &=
    \frac{(r-1)^2}{2r^3}\geq 0,
    \label{eq:psi-prime}\\
    \psi''
    &=
    \frac{(r-1)(3-r)}{4r^5}.
    \label{eq:psi-second}
\end{align}
Then \cref{eq:psi-combination-one} follows by direct computation. To prove $\psi>0,$ note $\psi'(0)=0, \quad \psi'> 0$ if $r>1$, and $\psi(0)=0.$ For \cref{eq:psi-combination-two}, direct simplification gives
\[
    (\psi')^2+\psi\psi''
    =
    \frac{r-1}{2r^6}
    \left[
        r(r-3)\log r+2(r-1)
    \right].
\]
Set
\[
    G(r)
    =
    r(r-3)\log r+2(r-1).
\]
Then
\[
    G(1)=G'(1)=0, \quad G''(r)
    =
    2\log r
    +
    3\left(1-\frac1r\right)
    \geq0
    \qquad
    (r\geq1).
\]
Hence \(G\geq0\), proving
\cref{eq:psi-combination-two}. The asymptotic expansion
\cref{eq:psi-large-s} is immediate from the definition.
\end{proof}

We now show that $K_t$ is positive in a uniform neighborhood of the boundary. 
\begin{lemma}[Uniform boundary positivity of \(K\)]
\label{lem:cp2-K-boundary}
Let \(\{\Omega_t\}\) be a smooth compact family of 1-convex domains. Let \(u_t\) be the normalized positive first
Dirichlet eigenfunction, let
\[
    f_t=-\log u_t,
    \qquad
    K_t
    =
    \nabla^2f_t
    -
    \psi(|\nabla f_t|^2)g.
\]
Then there exists a uniform boundary-collar radius
\(\rho_0>0\) such that
\[
    K_t>0
\]
whenever $  0<d(x,\partial\Omega_t)<\rho_0,
    \quad
    0\leq t\leq1.$
\end{lemma}

\begin{proof}
Write $\rho=d(\,\cdot\,,\partial\Omega_t)$
and let \(\nu=\nabla\rho\) be the inward unit normal. By the Hopf boundary
lemma and elliptic boundary regularity, near the boundary \cite{GilbargTrudinger2001}, we have
\[
    u_t=\rho a_t,
\]
where \(a_t>0\) is smooth and, because the deformation family is compact,
\(a_t\) and its derivatives are uniformly controlled.

Consequently,
\[
    f_t=-\log\rho-\log a_t
\]
and
\begin{equation}\label{eq:H-boundary-cp2}
    \nabla^2f_t
    =
    \frac{d\rho\otimes d\rho}{\rho^2}
    -
    \frac{\nabla^2\rho}{\rho}
    +O(1),
\end{equation}
uniformly in \(t\). For a unit vector $T$ tangent to a parallel equidistant hypersurface,
\[
    -\nabla^2\rho(T,T)
    =
    \II_{\partial\Omega_t}(T,T)+O(\rho).
\]
Thus
\begin{equation}\label{eq:H-boundary-tangent-cp2}
    \nabla^2f_t(T,T)
    =
    \frac{\II_{\partial\Omega_t}(T,T)}{\rho}
    +O(1).
\end{equation}

Furthermore,
\[
    |\nabla f_t|^2
    =
    \rho^{-2}+O(\rho^{-1}),
\]
and therefore, by
\cref{eq:psi-large-s},
\begin{equation}\label{eq:psi-boundary-cp2}
    \psi(|\nabla f_t|^2)
    =
    \frac1\rho
    +
    2\log\rho
    +
    O(1).
\end{equation}
For tangential \(T\), combining
\cref{eq:H-boundary-tangent-cp2,eq:psi-boundary-cp2} gives
\[
\begin{aligned}
    K_t(T,T)
    &=
    \frac{
        \II_{\partial\Omega_t}(T,T)-|T|^2
    }{\rho}
    -
    2\log\rho\,|T|^2
    +
    O(1)|T|^2\\
    &\geq
    2|\log\rho|\,|T|^2
    +
    O(1)|T|^2.
\end{aligned}
\]
This is strictly positive for sufficiently small \(\rho\), even in directions
where equality holds in \(\II\geq g\).

In the normal direction,
\[
    K_t(\nu,\nu)
    =
    \rho^{-2}-\rho^{-1}
    +O(|\log\rho|)>0,
\]
while the mixed normal--tangential entries are \(O(1)\). Since the mixed block is $O(1)$, the Schur complement criterion \cite[Chapter~1]{zhang2005schur} implies positivity for all sufficiently small $\rho$. Compactness of the family gives a uniform
choice of \(\rho_0\).
\end{proof}
\begin{remark}[Boundary asymptotics of the barrier]
The leading-order growth of $\psi$ is dictated by the $1$-convexity
assumption.
From preceding calculations, 
the leading-order growth of our barrier, $\frac{\psi(s)}{\sqrt{s}}\longrightarrow1,$ is optimal at the boundary.
Moreover, the lower-order term
$\psi(|\nabla f|^2)=\rho^{-1}+2\log\rho+O(1)$
provides the strict positivity needed even in directions where
$\mathrm{II}_{\partial\Omega}(T,T)=1$.
\end{remark}

\section{The first-contact calculation in $\CP^2$}\label{sec:The first-contact calculation in CP2}
In this section, we obtain the estimates needed to rule out interior null contact $(p,X)$ of $K.$
\subsection{A uniform spectral bound}
Before we carry out the first-contact calculations, we need a spectral bound $\lambda \geq 16$ of the first Dirichlet eigenvalue of $\Omega.$ We obtain this by combining an inradius estimate with a Riccati comparison argument.
\begin{lemma}[Spectral bound under $\kappa$-convexity]\label{lem:spectral bound kappa convex}
Let $\Omega\subset \CP^n$ be a connected smooth domain satisfying $\II_{\partial\Omega}\ge \kappa g$ for some $\kappa>0.$ Then
\[
    \lambda_1(\Omega)
    \ge
    \frac{2\pi^2}{\arctan^2(2/\kappa)}.
\]
In particular, when $\kappa=1$,
\[
    \lambda_1(\Omega)
    \ge
    \frac{2\pi^2}{\arctan^2 2}
    >16.
\]
\end{lemma}

\begin{proof}
Let
$  \rho(x)=d(x,\partial\Omega)$
be the distance to the boundary, and let
\[
    A_\rho=-\nabla^2\rho\big|_{(\nabla\rho)^\perp}
\]
be the shape operator of the inner parallel hypersurfaces, with respect to
the inward unit normal $\nu=\nabla\rho$.

Along a normal minimizing geodesic, the Riccati equation of the shape operator is
\[
    A_\rho'=A_\rho^2+R_\nu.
\]
Since the holomorphic sectional curvature of $\CP^n$ is $4$ and $J\nu$
is parallel along the normal geodesic, the function
\[
    h(\rho):=A_\rho(J\nu,J\nu)
\]
satisfies
\[
    h'
    =
    |A_\rho(J\nu)|^2+4
    \ge h^2+4,
    \qquad
    h(0)\ge\kappa.
\]
Comparison with the solution of
\[
    y'=y^2+4,
    \qquad
    y(0)=\kappa,
\]
gives
\[
    h(\rho)
    \ge
    2\tan\left(
        2\rho+\arctan\frac{\kappa}{2}
    \right).
\]
The comparison solution blows up at
\[
    R_\kappa
    :=
    \frac{\pi}{4}
    -\frac12\arctan\frac{\kappa}{2}
    =
    \frac12\arctan\frac{2}{\kappa}.
\]
A minimizing normal geodesic cannot pass through a focal point of
$\partial\Omega$. Hence $\rho(x)\le R_\kappa$
for every $x\in\Omega$.

We emphasize that this inradius estimate by itself does \emph{not} imply an
eigenvalue \emph{lower} bound by domain monotonicity. We now use the preceding Riccati comparison to
construct a positive supersolution.

Set
$ a_\kappa
    :=
    {\pi}/{2R_\kappa}
    =
    {\pi}/{\arctan(2/\kappa)}$
and define
$ w(x):=\sin(a_\kappa\rho(x)).$
Since $0<\rho\le R_\kappa$ in $\Omega$, we have $w>0$ in $\Omega$ and
$w=0$ on $\partial\Omega$.

At every regular point of $\rho$, the Riccati equation and
$A_0\ge\kappa I>0$ imply $A_\rho>0$. Therefore
\[
    -\Delta\rho
    =
    \tr A_\rho
    \ge h(\rho).
\]
Writing
$ \delta:=R_\kappa-\rho,$
the preceding comparison becomes
$ h(\rho)
    \ge
    2\cot(2\delta).$
Since $\kappa>0$, we have $R_\kappa<\pi/4$, hence $a_\kappa>2$. The
function
$   x\longmapsto x\cot x$
is decreasing on $(0,\pi/2)$. Since
\[
    0<2\delta\le a_\kappa\delta<\frac{\pi}{2},
\]
it follows that
\[
    2\cot(2\delta)
    \ge
    a_\kappa\cot(a_\kappa\delta).
\]
Because $a_\kappa R_\kappa=\pi/2$ and $\delta+\rho=R_\kappa$,
\[
    \cot(a_\kappa\delta)
    =
    \tan(a_\kappa\rho).
\]
Consequently,
\[
    \Delta\rho
    \le
    -a_\kappa\tan(a_\kappa\rho).
\]

At every regular point of $\rho$,
\[
\begin{aligned}
    \Delta w
    &=
    -a_\kappa^2\sin(a_\kappa\rho)
    +
    a_\kappa\cos(a_\kappa\rho)\Delta\rho\\
    &\le
    -a_\kappa^2 w
    -
    a_\kappa^2
    \cos(a_\kappa\rho)
    \tan(a_\kappa\rho)\\
    &=
    -2a_\kappa^2w.
\end{aligned}
\]
Thus
\[
    -\Delta w\ge 2a_\kappa^2w.
\]

The boundary-distance Laplacian comparison extends across the cut locus
in the barrier, hence distributional, sense; equivalently, the singular
part of $\Delta\rho$ at the cut locus has the favorable sign. Since
$t\mapsto\sin(a_\kappa t)$ is increasing and concave on
$[0,R_\kappa]$, the preceding inequality therefore holds weakly on all
of $\Omega$.

The weak form of Barta's inequality \cite{Barta1937} now gives
\[
    \lambda_1(\Omega)\ge 2a_\kappa^2.
\]
Indeed, if $\varphi\in C_c^\infty(\Omega)$, testing
$-\Delta w\ge2a_\kappa^2w$ against $\varphi^2/w$ and using
Cauchy--Schwarz yields
\[
    2a_\kappa^2\int_\Omega\varphi^2
    \le
    \int_\Omega |\nabla\varphi|^2.
\]
Taking the infimum of the Rayleigh quotient gives
\[
    \lambda_1(\Omega)
    \ge
    2a_\kappa^2
    =
    \frac{2\pi^2}{\arctan^2(2/\kappa)}.
\]
\end{proof}

\subsection{The $\CP^2$ first-contact calculation}

We now derive the interior inequality. The commutation identities used
below are precisely those recorded in
\cite[(2.3)--(2.4)]{KNTW}.

Let
\[
    \mathscr L
    :=
    \Delta-2\langle\nabla f,\nabla\cdot\rangle
\] be the drifted Laplacian.
\begin{lemma}[Bochner and null-direction identities for $K$]
\label{lem:cp2-contact-identities}
Let $s:= |\grad f|^2$. Suppose
\[
    K=H-\psi(s)g\geq0
\]
at an interior point \(p\), and let \(X \in T_p\CP^2\) be a unit null vector of \(K\).
Put
\[
    a=df(X),
    \qquad
    q=df(JX),
\]
and decompose
\[
    \nabla f=aX+qJX+Z,
    \qquad
    Z\perp X,JX.
\]
Put
$ B:=K|_{X^\perp}$ and $T:=\tr B.$
Then
\begin{align}
T
&= \lambda+s-4\psi,
\label{eq:T-contact-cp2}\\
|H|^2
&= 4\psi^2+2\psi T+|B|^2,
\label{eq:Hnorm-contact-cp2}\\
\frac12\mathscr Ls
&= |H|^2+6s.
\label{eq:bochner-s-cp2}
\end{align}
Set
\[
    \tau=q^2+|Z|^2,
    \qquad
    x=q^2,
    \qquad
    v=qJX+Z.
\]
Then, at \(p\),
\begin{equation}\label{eq:scalar-contact-E-cp2}
\begin{aligned}
    E
    :=
    \frac12\mathscr L K(X,X)
    &=
    \psi^2+4\psi-\lambda-a^2
    -3B(JX,JX)+3q^2\\
    &\quad
    -\psi'
    \left(
        4\psi^2+2\psi T+|B|^2+6s
    \right)\\
    &\quad
    -2\psi''
    \left(
        \psi^2a^2
        +
        |(\psi I+B)v|^2
    \right).
\end{aligned}
\end{equation}
\end{lemma}

\begin{proof}
    Since \(K(X,\cdot)=0\),
\[
    H(X,X)=\psi,
    \qquad
    H(X,V)=0
    \quad
    (V\perp X).
\]
The Bochner formula gives
\[
    \frac12\Delta s
    =
    |H|^2
    +
    \langle\nabla f,\nabla\Delta f\rangle
    +
    \Ric(\nabla f,\nabla f).
\]
Using
$ \Delta f=\lambda+s$
and \(\Ric=6g\), subtraction of
\(\langle\nabla f,\nabla s\rangle\) yields
\cref{eq:bochner-s-cp2}.
For the null-direction equation, we use the contact computation in the
proof of \cite[Theorem~2.1]{KNTW}. Applied to $b= \psi(s)=\psi(|\grad f|^2)$ it gives, at a null direction of $K=-\bigl(\nabla^2\log u+b\,g\bigr),$
\[
\frac12\mathscr L K(X,X)
=
b^2
+\operatorname{Tr}
\left[
\bigl(\operatorname{Hess}(\log u)+d(\log u)\otimes d(\log u)\bigr)\circ R_X
\right]
+6b
-\frac12\mathscr L b.
\]
Using
\[
R_X=I-X\otimes X^\flat+3JX\otimes JX^\flat,
\qquad
\Delta (\log u)+|\nabla (\log u)|^2=-\lambda,
\]
and
\[
\operatorname{Hess}(\log u)(X,X)=-\psi,
\qquad
\operatorname{Hess}(\log u)(JX,JX)
=-\psi-B(JX,JX),
\]
we obtain
\[
\operatorname{Tr}
\left[
\bigl(\operatorname{Hess}(\log u)+d(\log u)\otimes d(\log u)\bigr)\circ R_X
\right]
=
-\lambda-2\psi-a^2-3B(JX,JX)+3q^2.
\]
Moreover, the chain rule yields
\begin{align*}
\frac12\mathscr L b
=
\frac12\mathscr L\bigl(\psi(s)\bigr) 
&=
\frac12\left(
\psi'(s)\mathscr L s
+
\psi''(s)|\nabla s|^2
\right) \notag\\
&=
\psi'(s)\frac12\mathscr L s
+
\frac12\psi''(s)|\nabla s|^2 \notag\\
&=
\psi'(s)\left(|H|^2+6s\right)
+
\frac12\psi''(s)
\left|2H(\nabla f,\cdot)\right|^2 \notag\\
&=
\psi'(s)\left(|H|^2+6s\right)
+
2\psi''(s)|H(\nabla f,\cdot)|^2.
\end{align*}
Substituting \eqref{eq:T-contact-cp2}, \eqref{eq:Hnorm-contact-cp2}, and
\[
H(\nabla f,\cdot)
=
\psi a X^\flat+(\psi I+B)(qJX+Z)
\]
gives \eqref{eq:scalar-contact-E-cp2}.
\end{proof}

We now show that if an interior first contact occurs, the kernel of $K$ is one-dimensional.
\begin{lemma}[Simple nullity at an interior contact]\label{lem:Simple nullity at an interior contact}
    Let $u>0$ solve
$
-\Delta u=\lambda u
$
on an open subset of $ \CP^2$, where $\lambda \geq 16$. 
Suppose that $K \geq 0$ in a neighborhood of $p$ and that $K(p)$ is singular. Then
$$
\operatorname{dim} \operatorname{ker} K(p)=1 .
$$
In particular, $B=K|_{X^\perp}>0.$
\end{lemma}
\begin{proof}
    Suppose instead that $\operatorname{dim} \operatorname{ker} K(p) \geq 2$. The map
$$
\operatorname{ker} K(p) \longrightarrow \mathbb{R}, \quad W \longmapsto d f(J W),
$$
is linear, so there is a unit vector $X \in \operatorname{ker} K(p)$ such that
$$
q:=d f(J X)=0 .
$$
After extending $X$ locally with $\nabla X(p)=0$, the function $K(X, X)$ has a local minimum equal to zero at $p$. Hence
$$
E:=\frac{1}{2} \mathscr{L} K(X, X) \geq 0.
$$
First suppose $0 \leq s \leq 8$. Then $\psi^{\prime \prime} \geq 0$ by \cref{eq:psi-second}. Since $q=0, B \geq 0$, and $\psi^{\prime} \geq 0$, \cref{eq:scalar-contact-E-cp2} gives
$$
E \leq \psi^2+4 \psi-\lambda .
$$
The function $\psi$ is increasing and $\psi(8)<1 / 2$, so
$$
E<\frac{1}{4}+2-16<0,
$$
a contradiction.

Now suppose $s>8$, so that $\psi^{\prime \prime}<0$. Since $\operatorname{dim} \operatorname{ker} K(p) \geq 2$, the eigenvalues of $K(p)$ may be written as
$$
0, \quad 0, \quad \kappa, \quad T-\kappa,
$$
where $T=\operatorname{tr} K=\lambda+s-4 \psi$. Choose $\kappa$ to be the larger one of the last two eigenvalues. Then
\begin{equation}
\begin{gathered}
|H|^2=4 \psi^2+2 \psi T+\kappa^2+(T-\kappa)^2 \\
|H(\nabla f, \cdot)|^2 \leq s(\psi+\kappa)^2 .
\end{gathered}
\end{equation}

Using \cref{eq:scalar-contact-E-cp2} and dropping the favorable terms $-a^2$ and $-3 B(J X, J X)$, we obtain
$$
E \leq  \psi^2+4 \psi-\lambda -\psi^{\prime}\left(4 \psi^2+2 \psi T+\kappa^2+(T-\kappa)^2+6 s\right)-2 s \psi^{\prime \prime}(\psi+\kappa)^2.
$$ Then the right-hand side is a concave quadratic function of $\kappa$.
Completing the square gives the upper bound
$$
\begin{aligned}
E \leq & \psi^2+4 \psi-\lambda-2 \psi^{\prime} \psi^2-6 s \psi^{\prime} \\
& -2 \psi^{\prime} \frac{\psi^{\prime}+2 s \psi^{\prime \prime}}{\psi^{\prime}+s \psi^{\prime \prime}}\left(\frac{\lambda+s}{2}-\psi\right)^2 .
\end{aligned}
$$
Now $\psi''<0$ and \cref{eq:psi-combination-one} implies
$
\psi^{\prime}+s \psi^{\prime \prime}> \psi^{\prime}+2s \psi^{\prime \prime}>0.
$

Since $\lambda \geq 16$,
$$
\frac{\lambda+s}{2}-\psi \geq 8+\frac{s}{2}-\psi .
$$
The auxiliary barrier inequality (Lemma~\ref{lem:Auxiliary barrier inequality}) therefore yields
$$
E  \leq \psi^2+4 \psi-16-(\psi+2)^2-2 \psi^{\prime} \psi^2-6 s \psi^{\prime} 
 \leq -20-2 \psi^{\prime} \psi^2-6 s \psi^{\prime}
 <0
$$

This contradicts $E \geq 0$. Hence $\operatorname{ker} K(p)$ is one-dimensional.
\end{proof}

We next derive a \emph{negative} upper bound for $\mathscr L\det K.$
\begin{lemma}\label{lem: det reduction}
Assume the hypotheses and notation of Lemma~\ref{lem:cp2-contact-identities}. Suppose that $K\ge0$ in a neighborhood of $p$, that $K(p)$
has a one-dimensional kernel spanned by the unit vector $X$ so that $B>0$. For brevity, set
$$
\tau=q^2+|Z|^2,
\qquad
x=q^2.
$$
Then $0\le x\le\tau\le s = |\grad f|^2$, and
\begin{equation}\label{Ineq: L det K / 2det B}
    \begin{aligned}
\frac{\mathscr L\det K}{2\det B}+\lambda
\le {}&
\psi^2+4\psi
-\psi'\bigl(4\psi^2+2\psi T+6s\bigr)
-2\psi''\psi^2s-s \\
&+\Bigl[
1-4\Bigl(\psi'
+\sqrt{(\psi')^2+\psi\psi''}\Bigr)
(1+2\psi\psi')
\Bigr]\tau \\
&+\Bigl[
3-12\Bigl(\psi'
+\sqrt{(\psi')^2+\psi\psi''}\Bigr)
\Bigr]x
-\mathcal P(\tau,x),
\end{aligned}
\end{equation}
where
$$
\mathcal P(\tau,x)
:=
\min_{\substack{C=C^*\text{ on }X^\perp\\ \operatorname{tr}C=T}}
\left\{
\psi'|C|^2
+2\psi''|C(qJX+Z)|^2
+3C(JX,JX)
\right\}.
$$
Moreover, the right-hand side of \cref{Ineq: L det K / 2det B} attains its maximum over $0\le x\le\tau\le s$ at one of the three vertices
$$
(\tau,x)=(0,0),\quad (s,0),\quad (s,s).
$$
\end{lemma}

\begin{proof}

At a possible null contact one necessarily has $s= |\grad f|^2>0$, so $\psi'(s)>0$.

For $W\in T_p\CP^2$, define the covector on $X^\perp$ by

$$
C_W(V):=(\nabla_W K)(X,V).
$$
Choose a local orthonormal frame
\[
e_1=X,\qquad e_2,e_3,e_4\in X^\perp,
\qquad
\nabla e_j(p)=0,
\]
and write, near $p$,
\[
K=
\begin{pmatrix}
K_{11} & c^{T}\\
c & \widetilde B
\end{pmatrix},
\qquad
K_{11}=K(X,X).
\]
At the simple null contact,
\[
K_{11}(p)=0,\qquad c(p)=0,\qquad \widetilde B(p)=B>0.
\]
Hence the Schur complement formula gives
\[
\det K
=
\det\widetilde B \times\,
\bigl(K_{11}-c^{T}\widetilde B^{-1}c\bigr).
\]
Since $K_{11}$ has a local minimum at $p$, $\nabla K_{11}(p)=0$; together with
$c(p)=0$, this implies that both the zeroth- and first-order terms of the
Schur complement vanish at $p$. Therefore
\[
\Delta\det K
=
\det B\left(
\Delta  K_{11}
-
2\sum_j
\langle B^{-1}C_{e_j},C_{e_j}\rangle
\right),
\]
Indeed,
\[
\nabla^2_{e_j,e_j}
\bigl(c^{T}\widetilde B^{-1}c\bigr)(p)
=
2\langle B^{-1}C_{e_j},C_{e_j}\rangle,
\]
since every term involving a derivative of $\widetilde B^{-1}$ contains
a factor $c$ and hence vanishes at $p$. Moreover,
\[
\nabla K(X,X)(p)=0,
\qquad
\nabla\det K(p)=0,
\]
so the drift terms vanish at $p$. Consequently,
\[
\frac{\mathscr L\det K}{2\det B}
=
\frac12\mathscr L K(X,X)
-
\sum_j\langle B^{-1}C_{e_j},C_{e_j}\rangle
=
E-\sum_j\langle B^{-1}C_{e_j},C_{e_j}\rangle.
\]
% The block determinant formula gives
% $$
% \frac{\mathscr L\det K}{2\det B}
% =
% E-\sum_j\langle B^{-1}C_{e_j},C_{e_j}\rangle.
% $$

Thus the determinant test contributes additional favorable negative quadratic terms involving the mixed derivatives of $K$. Since every term in the determinant sum is nonnegative, retaining only the term corresponding to $W=X$ gives
$$
\frac{\mathscr L\det K}{2\det B}
\le
E-\langle B^{-1}C_X,C_X\rangle.
$$
The first-contact condition and the third-order derivative commutation formula give
$$
C_X
=
\left(
R_X|_{X^\perp}
+2\psi\psi' I
+2\psi'B
\right)(qJX+Z).
$$
Using
$$
R_X|_{X^\perp}
=
I+3JX\otimes JX^\flat,
$$
together with $(\psi')^2+\psi\psi''\ge0$, the Cauchy--Schwarz inequality yields
$$
\langle B^{-1}C_X,C_X\rangle
+4\bigl((\psi')^2+\psi\psi''\bigr)
\langle B(qJX+Z),qJX+Z\rangle 
\ge
4\left(
\psi'+\sqrt{(\psi')^2+\psi\psi''}
\right)
\left[
(1+2\psi\psi')\tau+3x
\right].
$$

For the next calculation, set
$$
M:=R_X|_{X^\perp}+2\psi\psi' I, \quad V: = qJX+Z.
$$
Then
$$
C_X=MV+2\psi'BV,
$$
and therefore
$$
\langle B^{-1}C_X,C_X\rangle
=
\langle B^{-1}MV,MV\rangle
+4\psi'\langle MV,V\rangle
+4(\psi')^2\langle BV,V\rangle.
$$
On the other hand,
$$
|(\psi I+B)V|^2
=
\psi^2\tau
+2\psi\langle BV,V\rangle
+|BV|^2.
$$

Substituting this into \cref{eq:scalar-contact-E-cp2}, and using
$
a^2=s-\tau,
\quad
q^2=x,
$
we obtain
$$
\begin{aligned}
E
={}&
\psi^2+4\psi-\lambda-s+\tau+3x
-\psi'\bigl(4\psi^2+2\psi T+6s\bigr)
-2\psi''\psi^2s \\
&-\psi'|B|^2
-2\psi''|BV|^2
-3B(JX,JX)
-4\psi\psi''\langle BV,V\rangle.
\end{aligned}
$$
Hence
$$
\begin{aligned}
\frac{\mathscr L\det K}{2\det B}+\lambda
\le {}&
\psi^2+4\psi-s
-\psi'\bigl(4\psi^2+2\psi T+6s\bigr)
-2\psi''\psi^2s
+\tau+3x \\
&-\Bigl[
\psi'|B|^2
+2\psi''|BV|^2
+3B(JX,JX)
\Bigr] \\
&-\Bigl[
\langle B^{-1}MV,MV\rangle
+4\bigl((\psi')^2+\psi\psi''\bigr)\langle BV,V\rangle
\Bigr]
-4\psi'\langle MV,V\rangle.
\end{aligned}
$$
Since $(\psi')^2+\psi\psi''\ge0$ (\cref{eq:psi-combination-two}),
we have
$$
\begin{aligned}
&\langle B^{-1}MV,MV\rangle
+4\bigl((\psi')^2+\psi\psi''\bigr)\langle BV,V\rangle \\
& \qquad \geq 4 \sqrt{(\psi')^2+\psi\psi''} \sqrt{\left\langle B^{-1} M V, M V\right\rangle\langle B V, V\rangle}\\
&\qquad\ge
4\sqrt{(\psi')^2+\psi\psi''}\,
\langle MV,V\rangle.
\end{aligned}
$$
The first inequality follows from AM-GM, and the second inequality follows from Cauchy--Schwarz,
$$
\langle B^{-1}MV,MV\rangle
\langle BV,V\rangle
\ge
\langle MV,V\rangle^2.
$$
Since
$
R_X|_{X^\perp}
=
I+3JX\otimes JX^\flat,
$
we have
$$
\langle MV,V\rangle
=
(1+2\psi\psi')\tau+3x.
$$
Consequently,
$$
\begin{aligned}
\frac{\mathscr L\det K}{2\det B}+\lambda
\le {}&
\psi^2+4\psi-s
-\psi'\bigl(4\psi^2+2\psi T+6s\bigr)
-2\psi''\psi^2s \\
&+\Bigl[
1-
4\Bigl(\psi'
+\sqrt{(\psi')^2+\psi\psi''}\Bigr)
(1+2\psi\psi')
\Bigr]\tau \\
&+\Bigl[
3-
12\Bigl(\psi'
+\sqrt{(\psi')^2+\psi\psi''}\Bigr)
\Bigr]x \\
&-\Bigl[
\psi'|B|^2
+2\psi''|BV|^2
+3B(JX,JX)
\Bigr].
\end{aligned}
$$

Finally, $B$ is a symmetric operator on $X^\perp$ with $\operatorname{tr}B=T$. By the definition of $\mathcal P(\tau,x)$,
$$
\psi'|B|^2
+2\psi''|BV|^2
+3B(JX,JX)
\ge
\mathcal P(\tau,x).
$$

Replacing the right-hand side by its minimum therefore gives inequality~\ref{Ineq: L det K / 2det B}.

It remains to maximize the upper bound of \cref{Ineq: L det K / 2det B} over
\[
0\le x\le \tau\le s.
\]
By Lemma~\ref{lem:quadratic-minimization},
\[
\frac{\partial^2\mathcal P}{\partial x^2}\le0,
\qquad
\frac{d^2}{d\tau^2}\mathcal P(\tau,0)\le0,
\qquad
\frac{d^2}{d\tau^2}\mathcal P(\tau,\tau)\le0.
\]
Thus $\mathcal P(\tau,\cdot)$ is concave, and its restrictions to the
two edges $x=0$ and $x=\tau$ are concave in $\tau$.

For fixed $s>0$, set
\[
\lambda_0=1-\frac{\tau}{s},
\qquad
\lambda_1=\frac{\tau-x}{s},
\qquad
\lambda_2=\frac{x}{s}.
\]
Then $\lambda_i\ge0$, $\lambda_0+\lambda_1+\lambda_2=1$, and the
preceding concavity gives
\[
\mathcal P(\tau,x)
\ge
\lambda_0\mathcal P(0,0)
+\lambda_1\mathcal P(s,0)
+\lambda_2\mathcal P(s,s).
\]
Since the remaining part of the upper bound is affine in $(\tau,x)$,
the full upper bound of \cref{Ineq: L det K / 2det B}, denoted by $\mathcal Q(\tau,x)$, satisfies
\[
\mathcal Q(\tau,x)
\le
\lambda_0\mathcal Q(0,0)
+\lambda_1\mathcal Q(s,0)
+\lambda_2\mathcal Q(s,s).
\]
Hence
\[
\mathcal Q(\tau,x)
\le
\max\{
\mathcal Q(0,0),\mathcal Q(s,0),\mathcal Q(s,s)
\}.
\]
Thus it suffices to consider the three extremal configurations
\[
(\tau,x)=(0,0),\quad (s,0),\quad (s,s).
\]
\end{proof}
The three extremal configurations maximizing right-hand side of \cref{Ineq: L det K / 2det B} have a clean geometric interpretation. Since
$$
\grad f= df(X)X+ df(JX)JX+Z, \quad
s=a^2+q^2+|Z|^2,
\qquad
\tau=q^2+|Z|^2,
\qquad
x=q^2,
$$
they correspond respectively to
\begin{equation}
\begin{gathered}
(\tau, x)=(0,0) \quad \Longleftrightarrow \quad \nabla f=a X, \\
(\tau, x)=(s, 0) \quad \Longleftrightarrow \quad \nabla f= Z \perp\{X, J X\}, \\
(\tau, x)=(s, s) \quad \Longleftrightarrow \quad \nabla f=q JX .
\end{gathered}
\end{equation}
The preceding lemma shows that for fixed $s= |\grad f|^2$, it suffices to consider the three extreme orientations of $\nabla f$ relative to the complex line spanned by $X$ and $JX$.

\begin{proposition}[No interior first-contact]\label{prop: no interior first contact}
    Assume $\lambda\ge16$. Then an interior semi-positive null contact $(p, X)$ of $K$ is impossible.
\end{proposition}

\begin{proof}
    We first exclude the range $0\leq s\le8$ from elementary properties of the scalar barrier (Lemma~\ref{lem:cp2-barrier-properties}). In this range $\psi''\ge0$. Since $K\geq 0$,
$$
q^2\le s,
\qquad
B(JX,JX)\ge0,
$$
and
$$
|H|^2\ge\frac{(\operatorname{tr}H)^2}{4}
=\frac{(\lambda+s)^2}{4}.
$$

Hence \cref{eq:scalar-contact-E-cp2} gives
$$
\frac12\mathscr LK(X,X)
\le
\psi^2+4\psi-\lambda+3s
-\psi'
\left(
\frac{(\lambda+s)^2}{4}+6s
\right).
$$

The right-hand side is decreasing in $\lambda$, so it is enough to set $\lambda=16$ by the spectral bound (Lemma~\ref{lem:spectral bound kappa convex}).
For $0\le s\le8$, we have $\psi''\ge0$, so $\psi'$ is increasing.
Since $\psi(0)=0$,
\[
\psi(s)=\int_0^s\psi'(t)\,dt\le s\psi'(s).
\]
Moreover, $\psi\le\psi(8)<\frac12$, and therefore
\[
\psi^2+4\psi\le\frac92\psi\le\frac92s\psi'.
\]
Using the preceding estimate with $\lambda=16$, we obtain
\[
\frac12\mathscr LK(X,X)
\le
-16+3s
-\psi'
\left(
\frac{(16+s)^2}{4}+\frac32s
\right).
\]
Set
\[
r=\sqrt{1+s},
\qquad
y=r-1\in[0,2].
\]
Since
\[
\psi'=\frac{(r-1)^2}{2r^3}= \frac{y^2}{2(y+1)^3},
\]
the right-hand side equals
\[
-\frac{
y^6-20y^5-78y^4-12y^3+472y^2+336y+128
}{
8(1+y)^3
}.
\]
For $0\le y\le2$,
\[
-20y^5-78y^4\ge-118y^4\ge-472y^2,
\qquad
-12y^3\ge-48y,
\]
so the numerator is bounded below by
\[
y^6+288y+128>0.
\]
Hence
\[
\frac12\mathscr LK(X,X)<0
\]
for all $0\le s\le8$.

It remains to consider $s\ge8$. In this case $\psi''\le0$. By Lemma~\ref{lem:Simple nullity at an interior contact}, the nullity is simple, hence $B>0.$ By Lemma~\ref{lem: det reduction}, at a simple null contact the upper bound for
$$
\frac{\mathscr L\det K}{2\det B}+\lambda
$$

attains its maximum at one of
$$
(\tau,x)=(0,0),\qquad (s,0),\qquad (s,s).
$$

For each of these three configurations, substitution of
$
\psi(s)=\frac{s}{\sqrt{1+s}}-\log(1+s)
$
shows that the corresponding upper bound is strictly negative whenever $\lambda\ge16$ and $s\ge8$. The verification is a one-dimensional calculus problem. See Lemma~\ref{lem:three vertices ineq}.

Consequently,
$$
\frac{\mathscr L\det K}{2\det B}<0.
$$

At an interior simple null contact, however, $\det K$ has a local minimum equal to zero, and hence
$$
\mathscr L\det K=\Delta\det K\ge0,
$$
a contradiction.

Therefore no interior semipositive null contact of $K$ can occur.
\end{proof}

\section{Proof of the gradient-dependent Hessian estimate}\label{sec:Proof of first main theorem}
In this section we prove \cref{eq:main-gradient-bound}. Our strategy is similar to \cite{WeiXiao2025, KNTW}. Start with a 1-convex domain $\Omega$ in $\CP^2.$ We use mean curvature flow (MCF) to deform \(\Omega\) to a sufficiently small domain that is \(C^\infty\)-close to a geodesic ball.
\begin{lemma}[Preservation of \(\kappa\)-convexity by MCF]\label{lem:cpn-kappa-convex-MCF}
    Let $M_t\subset\CP^2$ be a smooth mean curvature flow of closed hypersurfaces, with the normal chosen so that the second fundamental form $\II$ is positive on convex hypersurfaces. If $\II > \kappa g$ at $t=0$, then $\II > \kappa g$ for every later time for which the flow exists.
\end{lemma}
\begin{proof}
Set
$$
S:=\II-\kappa g.
$$
Let $H= \tr \II$ denote the mean curvature of the hypersurface.
Suppose that positivity of $S$ fails for the first time at $(p,t_*)$. Choose an orthonormal frame diagonalizing $\II$ at this point with eigenvalues $\mu_1, \mu_2, \mu_3.$ Assume without loss of generality that $e_1$ is a null direction of $S$. Then
$$
\mu_1=\kappa,
\qquad
\mu_2,\mu_3\ge \kappa.
$$
Since $\CP^2$ is locally symmetric, the curvature-derivative terms in Huisken's evolution equation \cite{Huisken1986} vanish. Combining \cite[Theorem~3.4]{Huisken1986} with
$$
\frac{\partial}{\partial t}g_{ij}=-2H\II_{ij},
$$
and evaluating in the $e_1$-direction gives
$$
\left(\frac{\partial}{\partial t}-\Delta\right)S_{11}
=
\kappa(|\II|^2+\Ric(\nu,\nu))
+
2\sum_{\alpha=2}^3
(\mu_\alpha-\kappa)\sec(e_1,e_\alpha).
$$
On $\CP^2$,
$
\operatorname{Ric}(\nu,\nu)=6,
\quad
\sec\ge1,
$
so
$$
\left(\frac{\partial}{\partial t}-\Delta\right)S_{11}>0.
$$
At a first null point of the positive tensor $S$, however, the parabolic maximum principle gives
$$
\left(\frac{\partial}{\partial t}-\Delta\right)S_{11}\le0,
$$
a contradiction. Hence $S\geq 0$ is preserved.
\end{proof}

We next show that, near the extinction time, the mean curvature flow produces a sufficiently small nearly-round domain whose first Dirichlet eigenfunction satisfies \(K>0\).

\begin{lemma}[Initial log-concavity via MCF]
\label{lem:late-time-initialization}
Let $\{\Omega_t\}_{0\le t<T}$ be the domains enclosed by the mean
curvature flow starting from a smooth $1$-convex domain
$\Omega_0\subset\CP^2$. For each $t$, let $u_t>0$ be the first
Dirichlet eigenfunction of $\Omega_t$, and set
\[
f_t=-\log u_t,
\qquad
K_t=\nabla^2 f_t-\psi(|\nabla f_t|^2)g.
\]
Then
\[
K_t>0
\]
on $\Omega_t$ for all $t$ sufficiently close to the extinction time
$T$.
\end{lemma}

\begin{proof}
Since $\CP^2$ has nonnegative sectional curvature and $\nabla R=0$,
Huisken's convergence theorem \cite[Theorem~1.1]{Huisken1986} applies to the strictly convex
hypersurfaces $\partial\Omega_t$. Thus the flow contracts in finite time to
a point $p\in\CP^2$, and, after rescaling in normal coordinates about
$p$, the hypersurfaces converge smoothly to a round sphere.

More precisely, there exist scales $\varepsilon_t\to0$ such that,
after adjusting $\varepsilon_t$ by a fixed constant, the rescaled
hypersurfaces
\[
\widehat M_t
:=
\varepsilon_t^{-1}
\exp_p^{-1}(\partial\Omega_t)
\subset T_p\CP^2\simeq\mathbb R^4
\]
converge in $C^\infty$ to the unit sphere $\mathbb S^3$. In
particular, for $t$ sufficiently close to $T$, $\widehat M_t$ is a
small normal graph over $\mathbb S^3$. If $\widehat\Omega_t$ denotes
the region enclosed by $\widehat M_t$, there are diffeomorphisms from the unit ball $B_1$ to $\widehat\Omega_t$
\[
\Psi_t:\overline{B_1}\longrightarrow\overline{\widehat\Omega_t}
\]
such that
\[
\Psi_t\longrightarrow\Id
\qquad\text{in }C^\infty(\overline{B_1}).
\]
Indeed, the normal-graph parametrizations near $\partial B_1$ may be
extended to the enclosed domains by taking the identity on a smaller
ball and interpolating in a fixed collar.

Define
\[
\Phi_t
:=
\exp_p\circ D_{\varepsilon_t}\circ\Psi_t:
\overline{B_1}\longrightarrow\overline{\Omega_t},
\qquad
D_{\varepsilon_t}(x)=\varepsilon_t x,
\]
and put
\[
\widetilde g_t
:=
\varepsilon_t^{-2}\Phi_t^*g.
\]
In normal coordinates centered at $p$,
\[
g_{ij}(x)=\delta_{ij}+O(|x|^2),
\]
with the corresponding estimates for all derivatives. Hence
\[
\varepsilon_t^{-2}
(\exp_p\circ D_{\varepsilon_t})^*g
\longrightarrow g_{\mathrm E}
\qquad\text{in }C^\infty
\]
on every fixed bounded set. Since $\Psi_t\to\Id$ smoothly, it follows
that
\begin{equation}
\label{eq:rescaled-metric-convergence}
\widetilde g_t
\longrightarrow g_{\mathrm E}
\qquad\text{in }C^\infty(\overline{B_1}).
\end{equation}

Let
\[
\widetilde u_t:=u_t\circ\Phi_t,
\qquad
\widetilde f_t:=-\log\widetilde u_t.
\]
After normalizing $\widetilde u_t$ in $L^2(B_1,\widetilde g_t)$, it
satisfies
\[
-\Delta_{\widetilde g_t}\widetilde u_t
=
\widetilde\lambda_t\widetilde u_t,
\qquad
\widetilde u_t|_{\partial B_1}=0,
\qquad
\widetilde\lambda_t
=
\varepsilon_t^2\lambda_1(\Omega_t).
\]
By \eqref{eq:rescaled-metric-convergence}, simplicity of the first
Dirichlet eigenvalue, and standard elliptic boundary regularity,
\[
\widetilde u_t\longrightarrow u_{\mathrm E}
\qquad\text{in }C^3(\overline{B_1}),
\]
where $u_{\mathrm E}$ is the normalized first Dirichlet eigenfunction
of the Euclidean unit ball.

We claim that there is a constant $c>0$, independent of $t$ sufficiently
close to $T$, such that
\begin{equation}
\label{eq:rescaled-log-hessian}
\nabla_{\widetilde g_t}^2\widetilde f_t
\ge
c\bigl(1+|\nabla\widetilde f_t|_{\widetilde g_t}\bigr)
\widetilde g_t
\qquad\text{on }B_1.
\end{equation}
To see this near the boundary, use the fixed defining function
\[
\rho(x)=1-|x|^2.
\]
By the Hopf boundary lemma and elliptic boundary regularity,
\[
\widetilde u_t=\rho a_t,
\]
where, uniformly for $t$ close to $T$,
\[
a_t\ge c_0>0,
\qquad
\|a_t\|_{C^2(\overline{B_1})}\le C_0.
\]
Since
\[
-\nabla_{g_{\mathrm E}}^2\rho=2g_{\mathrm E},
\]
the convergence \eqref{eq:rescaled-metric-convergence} implies
\[
-\nabla_{\widetilde g_t}^2\rho\ge\widetilde g_t
\]
for $t$ sufficiently close to $T$. Therefore
\[
\nabla_{\widetilde g_t}^2\widetilde f_t
=
\frac{d\rho\otimes d\rho}{\rho^2}
-
\frac{\nabla_{\widetilde g_t}^2\rho}{\rho}
-
\nabla_{\widetilde g_t}^2\log a_t
\ge
\left(\frac1\rho-C\right)\widetilde g_t,
\]
while
\[
|\nabla\widetilde f_t|_{\widetilde g_t}
\le
\frac{C}{\rho}.
\]
Hence \eqref{eq:rescaled-log-hessian} holds in a fixed boundary
collar.

On the complement of this collar, $u_{\mathrm E}$ is bounded away
from zero. The Euclidean first eigenfunction on a Euclidean ball is strictly
log-concave, so
\[
\nabla_{g_{\mathrm E}}^2(-\log u_{\mathrm E})>0.
\]
The $C^3$-convergence above then gives
\eqref{eq:rescaled-log-hessian} on the remaining compact subset,
after decreasing $c$ if necessary.

Set
\[
\widetilde w_t
=
|\nabla\widetilde f_t|_{\widetilde g_t}.
\]
Since $\Phi_t^*g=\varepsilon_t^2\widetilde g_t$, constant rescaling
does not change the Levi--Civita connection, and therefore
\[
\Phi_t^*K_t
=
\nabla_{\widetilde g_t}^2\widetilde f_t
-
\varepsilon_t^2
\psi\bigl(\varepsilon_t^{-2}\widetilde w_t^2\bigr)
\widetilde g_t.
\]
From the definition of $\psi$,
\[
0\le\psi(s)\le\sqrt{s}
\qquad(s\ge0).
\]
Thus
\[
\varepsilon_t^2
\psi\bigl(\varepsilon_t^{-2}\widetilde w_t^2\bigr)
\le
\varepsilon_t\widetilde w_t.
\]
Using \eqref{eq:rescaled-log-hessian}, we obtain
\[
\Phi_t^*K_t
\ge
\left[
c(1+\widetilde w_t)
-\varepsilon_t\widetilde w_t
\right]\widetilde g_t.
\]
Since $\varepsilon_t\to0$, for all $t$ sufficiently close to $T$ we
have $\varepsilon_t\le c/2$, and hence
\[
\Phi_t^*K_t
\ge
c\,\widetilde g_t>0.
\]
Therefore $K_t>0$ on $\Omega_t$ for all sufficiently late times.
\end{proof}

% We follow the same strategy as \cite{KNTW}. Evolve \(\partial\Omega\) by mean curvature flow. By Huisken's convex contraction theorem
% \cite{Huisken1986}, the flow contracts in finite time to a point and
% becomes asymptotically round. Stop the flow sufficiently close to the
% extinction time and reverse the resulting smooth family. This gives a
% smooth deformation
% $\{\Omega_t\}_{0\le t\le1},$
% where $\Omega_1=\Omega$
% and \(\Omega_0\) is sufficiently small and \(C^\infty\)-close to a
% geodesic ball. Lemma~\ref{lem:cp2-K-small-ball} shows $K>0$ for small balls; and $K>0$ is preserved under smooth perturbation. Hence $K>0$ on $\Omega_0.$

We now complete the deformation argument. Starting from $\Omega$, evolve
$\partial\Omega$ by mean curvature flow and denote the enclosed domains by
\[
\{\Omega_\tau\}_{0\le \tau<T},
\qquad
\Omega_0=\Omega,
\]
where $T$ is the extinction time. By Lemma~\ref{lem:late-time-initialization}, we may choose
$\tau_0<T$ sufficiently close to $T$ so that the modified Hessian on
$\Omega_{\tau_0}$ is positive definite.

We then traverse the finite smooth family
$\{\Omega_\tau\}_{0\le\tau\le\tau_0}$ in the reverse direction. More
precisely, define
\[
\widetilde\Omega_t
:=
\Omega_{(1-t)\tau_0},
\qquad
0\le t\le1.
\]
Thus
\[
\widetilde\Omega_0=\Omega_{\tau_0},
\qquad
\widetilde\Omega_1=\Omega.
\]
This is only a reversal of the already constructed smooth family, not a
backward mean curvature flow. By Lemma~\ref{lem:cpn-kappa-convex-MCF}, every
$\widetilde\Omega_t$ remains $1$-convex. Relabeling
$\widetilde\Omega_t$ as $\Omega_t$, we therefore obtain a smooth
deformation
\[
\{\Omega_t\}_{0\le t\le1},
\qquad
\Omega_1=\Omega,
\]
such that $K_0>0.$

For each $t\in[0,1]$, let $u_t>0$ be the normalized first Dirichlet
eigenfunction of $\Omega_t$, and set
\[
f_t=-\log u_t,
\qquad
s_t=|\nabla f_t|^2,
\qquad
K_t=\nabla^2f_t-\psi(s_t)g.
\]
By Lemma~\ref{lem:cp2-K-boundary},
$  K_t>0$
in a uniform boundary collar for every \(t\).
Now Lemma~\ref{lem:spectral bound kappa convex} gives $\lambda_t\geq16$ for every $t$, and we can then apply all the interior estimate lemmas thereafter.

Suppose that positivity of \(K_t\) fails. By simplicity of the first
Dirichlet eigenvalue, standard elliptic perturbation theory gives smooth
dependence of \(u_t\), after pulling the domains back to a fixed reference
domain. Thus there is a first time \(t_*\) at which
$  K_{t_*}\geq0$
but \(K_{t_*}\) has a nontrivial kernel. The uniform boundary estimate puts
the contact point in the interior. However, Proposition~\ref{prop: no interior first contact} rules out any interior first-contact pair $(p,X).$ This proves \cref{eq:main-gradient-bound} by contradiction.

\section{A uniform Hessian lower bound and applications}
\label{sec:uniform-hessian-bound}

The estimate \cref{eq:main-gradient-bound} is strongest when
$s=|\grad f|^2$ is large, while its right-hand side vanishes at $s=0$.
We now combine it with the spectral bound in
\cref{lem:spectral bound kappa convex} and the determinant argument of
\cref{lem: det reduction} to obtain a uniform lower bound for $H=\nabla^2 f$. Together with \cref{eq:main-gradient-bound}, we complete the \cref{thm: main}.

\begin{theorem}[Uniform Hessian lower bound]
\label{prop:uniform-Hessian-lower-bound}
Let $\Omega\subset\CP^2$ be a connected $1$-convex domain with smooth
boundary, and let $u>0$ be its first Dirichlet eigenfunction and set $f:=-\log u$. Then
\[
    \nabla^2 f>\frac85 g
    \qquad\text{in }\Omega.
\]
\end{theorem}

\begin{proof}
By \cref{eq:main-gradient-bound} and the boundary expansions in the proof
of \cref{lem:cp2-K-boundary}, the least eigenvalue of $H$ tends to infinity
as one approaches $\partial\Omega$. Hence it attains
its minimum at an interior point $p\in\Omega$. Denote this minimum by
$\eta$, and choose a unit vector $X\in T_p\CP^2$ such that
\[
    H(X,\cdot)=\eta X^\flat.
\]
Suppose, to the contrary, that $\eta\le 8/5$, and set
\[
    K_\eta:=H-\eta g\geq0.
\]

We first show that $K_\eta(p)$ has one-dimensional kernel. If its
nullity were at least two, then the linear map
\[
    \ker K_\eta(p)\longrightarrow\R,
    \qquad
    Y\longmapsto df(JY),
\]
would have a nontrivial kernel. Thus we may choose a unit null vector
$X$ with $q:=df(JX)=0$. Put $a:=df(X)$ and
$B:=K_\eta|_{X^\perp}$. Repeating the null-direction calculation in
\cref{lem:cp2-contact-identities} with the constant barrier $\eta$
gives
\[
    0\leq \frac12\mathscr L K_\eta(X,X)
    =
    \eta^2+4\eta-\lambda-a^2-3B(JX,JX)
    \leq
    \eta^2+4\eta-\lambda<0,
\]
where the last inequality follows from
\cref{lem:spectral bound kappa convex} that $\lambda> 16$. This is impossible. Hence the
nullity is simple, so $B>0$.

We may therefore apply the determinant argument from the proof of
\cref{lem: det reduction}. Decompose
\[
    \grad f=aX+qJX+Z,
    \qquad
    Z\perp X,JX,
\]
and set
\[
    Y=qJX+Z.
\]
Since the barrier is constant, the mixed derivative term in the
determinant calculation reduces to $R_XY$. Since
$\det K_\eta$ has a local minimum equal to zero at $p$,

\begin{equation}\label{eq:uniform-Hessian-determinant}
\begin{aligned}
0
\leq
\frac{\mathscr L\det K_\eta}{2\det B}
\leq{}&
\eta^2+4\eta-\lambda-a^2
-3B(JX,JX)+3q^2 
-\langle B^{-1}R_XY,R_XY\rangle.
\end{aligned}
\end{equation}

By \cref{FS curvature},
\[
    \langle R_XY,JX\rangle=4q.
\]
Thus Cauchy--Schwarz gives
\[
    \langle B^{-1}R_XY,R_XY\rangle
    B(JX,JX)
    \geq16q^2,
\]
Since $B>0$, we have $B(JX,JX)>0$. Hence, by the arithmetic--geometric
mean inequality,
\[
3B(JX,JX)
+\frac{16q^2}{B(JX,JX)}
\ge
2\sqrt{48q^2}
=
8\sqrt3\,|q|.
\]
It follows from \cref{eq:uniform-Hessian-determinant} that
\begin{equation}\label{eq:uniform-Hessian-q-bound}
    0
    \leq
    \eta^2+4\eta-\lambda
    +3q^2-8\sqrt3\,|q|.
\end{equation}

On the other hand, \cref{eq:main-gradient-bound} gives
\[
    \psi(s)<H(X,X)=\eta\leq\frac85.
\]
Since $\psi$ is increasing by \cref{eq:psi-prime} and
\[
    \psi(25)
    =
    \frac{25}{\sqrt{26}}-\log26
    >
    \frac85,
\]
we have $s=|\grad f|^2<25$, and hence $|q|\leq|\grad f|<5$. The function
\[
    x\longmapsto3x^2-8\sqrt3\,x
\]
is convex on $[0,5]$, so
\[
    3q^2-8\sqrt3\,|q|
    \leq
    75-40\sqrt3.
\]
Using again $\lambda>16$ and $\eta\leq8/5$ in
\cref{eq:uniform-Hessian-q-bound}, we obtain
\[
    0
    \leq
    \eta^2+4\eta-\lambda
    +3q^2-8\sqrt3\,|q|
    \leq
    \frac{1699}{25}-40\sqrt3
    <0,
\]
a contradiction. Therefore $\eta>8/5$ and the theorem follows. 
\end{proof}
In the proof above, the role of \cref{eq:main-gradient-bound} is to control the large-gradient
regime: at a point where the least eigenvalue of $\nabla^2f$ is at most $8/5$, it forces $|\nabla f|^2<25$. The remaining bounded-gradient regime can then be excluded using the $\lambda_1$ lower bound and a second determinant maximum principle argument.
\begin{remark}
The constant $8/5$ in
\cref{prop:uniform-Hessian-lower-bound} is chosen for simplicity rather
than optimality. If $\eta$ denotes the minimum eigenvalue of
$\nabla^2 f$, the proof shows that any constant $\eta_0>0$ is admissible
provided
\[
\eta_0^2+4\eta_0-\lambda_*
+
\max_{0\leq x\leq r(\eta_0)}
\bigl(3x^2-8\sqrt3\,x\bigr)
<0,
\]
where $\lambda_*:=\frac{2\pi^2}{\arctan^2 2}$ by Lemma~\ref{lem:spectral bound kappa convex}
and $r(\eta_0)>0$ is determined implicitly by
\[
\psi\bigl(r(\eta_0)^2\bigr)=\eta_0.
\]
Thus the supremum of the admissible values of $\eta_0$ gives the best constant
provided by this sufficient condition. We use $8/5$ because
$\psi(25)>8/5$, so that the proof reduces to the simple bound
$|\langle\nabla f,JX\rangle|<5$.
\end{remark}

Combining \cref{prop:uniform-Hessian-lower-bound} with the
gradient-dependent estimate \cref{eq:main-gradient-bound} gives a uniform
quantitative convexity estimate for the level sets.

\begin{corollary}
\label{cor:uniform-level-set-convexity}
Let $\Omega\subset\CP^2$ be a connected $1$-convex domain with smooth
boundary, and let $u>0$ be its first Dirichlet eigenfunction. Then every regular superlevel set
\[
    \Omega_c
    =
    \{x\in\Omega:u(x)>c\},
    \qquad
    0<c<\max_\Omega u,
\]
has uniformly convex boundary. More precisely, along the regular level
set $\Sigma_c=\{u=c\}$,
\[
    \II_{\Sigma_c}
    >
    \frac{8}{25}g|_{T\Sigma_c}.
\]
\end{corollary}

\begin{proof}
Let $s=|\grad f|^2>0$. For $V\in T\Sigma_c$, our convention for
$\II$ gives
\[
    \II_{\Sigma_c}(V,V)
    =
    \frac{H(V,V)}{\sqrt{s}}.
\]
If $s\leq25$, then \cref{prop:uniform-Hessian-lower-bound} yields
\[
    \frac{H(V,V)}{\sqrt{s}}
    >
    \frac{8}{5\sqrt{s}}|V|^2
    \geq
    \frac{8}{25}|V|^2.
\]

We next observe that $\psi(s)/\sqrt{s}$ is increasing for $s>0$.
Indeed,
\[
    \frac{d}{ds}\left(\frac{\psi(s)}{\sqrt{s}}\right)
    =
    \frac{2s\psi'(s)-\psi(s)}{2s^{3/2}},
\]
while
\[
    \frac{d}{ds}
    \bigl(2s\psi'(s)-\psi(s)\bigr)
    =
    \psi'(s)+2s\psi''(s)>0
\]
by \cref{eq:psi-combination-one}, and
$2s\psi'(s)-\psi(s)\to0$ as $s\to0$.
Hence, if $s\geq25$, \cref{eq:main-gradient-bound} gives
\[
    \frac{H(V,V)}{\sqrt{s}}
    >
    \frac{\psi(s)}{\sqrt{s}}|V|^2
    \geq
    \frac{\psi(25)}5|V|^2
    >
    \frac{8}{25}|V|^2.
\]
\end{proof}

Finally, the classical ground-state transform of
Singer--Wong--Yau--Yau \cite{SingerWongYauYau1985}, together with its
Bakry--\'Emery interpretation due to Lu--Rowlett \cite{Lu2012Eigenvalues} yields the following fundamental gap estimate.

\begin{corollary}
\label{cor:fundamental-gap-uniform-Hessian}
Let $\Omega\subset\CP^2$ be a connected $1$-convex domain with smooth
boundary, and let $u>0$ be its first Dirichlet eigenfunction. The first two Dirichlet
eigenvalues satisfy
\[
    \lambda_2-\lambda_1>\frac{46}{5}.
\]
\end{corollary}

\begin{proof}
For completeness we present the proof here.
Let $u_1=u$ be the first Dirichlet eigenfunction and let $u_2$ be a
second Dirichlet eigenfunction. Following
Singer--Wong--Yau--Yau \cite{SingerWongYauYau1985}, set
$ w:=u_2/u_1.$
Then $w$ extends smoothly to $\overline{\Omega}$, satisfies the Neumann
boundary condition on $\partial\Omega$, and
\[
    -\mathscr L w=(\lambda_2-\lambda_1)w,
    \qquad
    \mathscr L
    =
    \Delta-2\langle\nabla f,\nabla\cdot\rangle,
\]
where $f=-\log u_1$. Equivalently, $\lambda_2-\lambda_1$ is the first
nonzero Neumann eigenvalue of $\mathscr L$ with respect to the weighted
measure $ u_1^2\,dV_g=e^{-2f}\,dV_g.$
The corresponding Bakry--\'Emery Ricci tensor is
$ \Ric_{2f}
    =
    \Ric+2\nabla^2f.$
Combining $\Ric=6g$ with \cref{prop:uniform-Hessian-lower-bound}
gives
\[
    \Ric_{2f}=\Ric+2\nabla^2 f
    >
    \frac{46}{5}g.
\]
The weighted Bochner formula gives
\[
    \frac12\mathscr L|\nabla w|^2
    =
    |\nabla^2w|^2
    +
    \langle\nabla w,\nabla\mathscr Lw\rangle
    +
    (\Ric+2\nabla^2f)(\nabla w,\nabla w).
\]
Integrating against $u_1^2\,dV_g$ and using the Neumann boundary
condition yields
\[
\begin{aligned}
(\lambda_2-\lambda_1)
\int_\Omega |\nabla w|^2u_1^2\,dV_g
={}&
\int_\Omega |\nabla^2w|^2u_1^2\,dV_g 
+
\int_\Omega
(\Ric+2\nabla^2f)(\nabla w,\nabla w)
u_1^2\,dV_g \\
>{}&
\frac{46}{5}
\int_\Omega |\nabla w|^2u_1^2\,dV_g.
\end{aligned}
\]
Since $w$ is nonconstant, we conclude that
$ \lambda_2-\lambda_1>\frac{46}{5}.$
\end{proof}

\appendix
\section{Technical lemmas}
In this appendix we collect the proofs of several technical lemmas in the proof of Theorem~\ref{thm: main}. The first one is the auxiliary barrier inequality used in Lemma~\ref{lem:Simple nullity at an interior contact}.
\begin{lemma}[Auxiliary barrier inequality]\label{lem:Auxiliary barrier inequality}
For $s \geq 8$,
$$
2 \psi^{\prime} \frac{\psi^{\prime}+2 s \psi^{\prime \prime}}{\psi^{\prime}+s \psi^{\prime \prime}}\left(8+\frac{s}{2}-\psi\right)^2 \geq(\psi+2)^2 .
$$    
\end{lemma}
\begin{proof}
    Set $r=\sqrt{1+s} \geq 3 .$
From the explicit formulas for $\psi^{\prime}$ and $\psi^{\prime \prime}$,
$$
2 \psi^{\prime} \frac{\psi^{\prime}+2 s \psi^{\prime \prime}}{\psi^{\prime}+s \psi^{\prime \prime}}=\frac{2(r-1)^2(2 r+3)}{r^3\left(r^2+2 r+3\right)} \geq \frac{4(r-1)^2}{r^2(r+1)^2} .
$$
It therefore suffices to prove
$$
\frac{2(r-1)}{r(r+1)}\left(8+\frac{s}{2}-\psi\right) \geq \psi+2 .
$$
The difference between the two sides is
$
\frac{2 N(r)}{r^2(r+1)},
$
where
$$
N(r)=\left(r^3+3 r^2-2 r\right) \log r-3 r^3+8 r^2-6 r-1 .
$$

Now
$
N(3)=48 \log 3-28>0
$
and
$$
N^{\prime}(r)=\left(3 r^2+6 r-2\right) \log r-8 r^2+19 r-8 .
$$

Writing $y=\frac{r-1}{r+1},$
we have
$$
\log r=2 \operatorname{arctanh} y \geq 2\left(y+\frac{y^3}{3}\right) .
$$
Consequently,
$$
N^{\prime}(r) \geq \frac{33 r^4+59 r^3+51 r^2-63 r-8}{3(r+1)^3}>0
$$
for $r \geq 3$. Hence $N(r)>0$, which proves the auxiliary inequality.
\end{proof}

We next compute \[\mathcal P(\tau,x)
:=
\min_{\substack{C=C^*\text{ on }X^\perp\\ \operatorname{tr}C=T}}
\left\{
\psi'|C|^2
+2\psi''|C(qJX+Z)|^2
+3C(JX,JX)
\right\}\] in the proof of Lemma~\ref{lem: det reduction}.

\begin{lemma}[The minimization in Lemma~4.4]
\label{lem:quadratic-minimization}
Fix $s>0$, and regard $\psi'=\psi'(s)$ and $\psi''=\psi''(s)$ as
constants. Assume
\[
\psi'>0,
\qquad
\psi'+2s\psi''>0.
\]
Let
\[
0\le x\le\tau\le s,
\qquad
V=qJX+Z\in X^\perp,
\]
where
\[
|V|^2=\tau,
\qquad
\langle V,JX\rangle^2=x,
\]
and fix $T\in\mathbb R$. Define
\[
\mathcal P(\tau,x)
:=
\min_{\substack{C=C^*\text{ on }X^\perp\\ \tr C=T}}
\left\{
\psi'|C|^2
+
2\psi''|CV|^2
+
3C(JX,JX)
\right\}.
\]
Then the minimum is unique. For $0<\tau\le s$,
\begin{align}
\mathcal P(\tau,x)
={}&
\frac{
2T^2(\psi')^2(\psi'+2\tau\psi'')
+
6T\psi'(\psi'+2\tau\psi'')
-
9(\psi'+\tau\psi'')
}
{2\psi'(3\psi'+4\tau\psi'')}
\notag\\
&-
\frac{
3\psi''
\left(
4T\tau\psi'\psi''
+
4T(\psi')^2
-
6\tau\psi''
-
3\psi'
\right)
}
{
2\psi'(\psi'+\tau\psi'')
(3\psi'+4\tau\psi'')
}
\,x
\notag\\
&-
\frac{
9(\psi'')^2
}
{
2\psi'(\psi'+\tau\psi'')
(3\psi'+4\tau\psi'')
}
\,x^2.
\label{eq:P-explicit}
\end{align}
Moreover,
\begin{equation}
\label{eq:P-xx}
\frac{\partial^2\mathcal P}{\partial x^2}
=
-\frac{
9(\psi'')^2
}{
\psi'(\psi'+\tau\psi'')
(3\psi'+4\tau\psi'')
}
\le0,
\end{equation}
and
\begin{align}
\frac{d^2}{d\tau^2}\mathcal P(\tau,0)
&=
-\frac{
4(\psi'')^2(2T\psi'+3)^2
}{
(3\psi'+4\tau\psi'')^3
}
\le0,
\label{eq:P-tau-0}\\
\frac{d^2}{d\tau^2}\mathcal P(\tau,\tau)
&=
-\frac{
16(\psi'')^2(T\psi'-3)^2
}{
(3\psi'+4\tau\psi'')^3
}
\le0.
\label{eq:P-tau-tau}
\end{align}
The value at $(\tau,x)=(0,0)$ is obtained by continuity.
\end{lemma}

\begin{proof}
We first check that the minimization problem is strictly convex. The tangent
space to the affine constraint $\tr C=T$ consists of trace-free symmetric
tensors $D$. Along such a variation,
\[
\frac12\frac{d^2}{dt^2}\Big|_{t=0}
\left[
\psi'|C+tD|^2
+2\psi''|(C+tD)V|^2
+3(C+tD)(JX,JX)
\right]
=
\psi'|D|^2+2\psi''|DV|^2.
\]
If $\psi''\ge0$, this is strictly positive for every $D\neq0$ since
$\psi'>0$. If $\psi''<0$, then
\[
|DV|^2\le |D|^2|V|^2=\tau|D|^2,
\]
and hence
\[
\psi'|D|^2+2\psi''|DV|^2
\ge
(\psi'+2\tau\psi'')|D|^2
\ge
(\psi'+2s\psi'')|D|^2>0.
\]
Thus the Hessian of the functional is positive definite on the tangent space
of the constraint, so the functional is coercive and strictly convex on
$\{\tr C=T\}$. In particular, the minimizer is unique.

We also note that
\[
\psi'+\tau\psi''>0,
\qquad
3\psi'+4\tau\psi''
=
\psi'+2(\psi'+2\tau\psi'')>0.
\]

Assume $\tau>0$. Since the functional is unchanged if $V$ is replaced
by $-V$, choose an orthonormal basis
$\{e_1,e_2,e_3\}$ of $X^\perp$ such that
\[
V=\sqrt{\tau}\,e_1,
\qquad
JX
=
\sqrt{\frac{x}{\tau}}\,e_1
+
\sqrt{\frac{\tau-x}{\tau}}\,e_2.
\]
Write $C=(C_{ij})$
in this basis. Using
$C_{33}=T-C_{11}-C_{22},$
the functional becomes
\begin{align*}
&\psi'\left(
C_{11}^2+C_{22}^2+(T-C_{11}-C_{22})^2
+2C_{12}^2+2C_{13}^2+2C_{23}^2
\right)\\
&\quad
+
2\tau\psi''
\left(
C_{11}^2+C_{12}^2+C_{13}^2
\right)\\
&\quad
+
3\left(
\frac{x}{\tau}C_{11}
+
\frac{2\sqrt{x(\tau-x)}}{\tau}C_{12}
+
\frac{\tau-x}{\tau}C_{22}
\right).
\end{align*}
Regarding the five free entries of $C=(C_{ij})$ as variables, Differentiating the functional with respect to its variables and setting the derivatives equal to zero gives
\[
C_{13}=C_{23}=0,
\qquad
C_{12}
=
-\frac{
3\sqrt{x(\tau-x)}
}{
2\tau(\psi'+\tau\psi'')
},
\]
and
\begin{align}
(\psi'+2\tau\psi'')C_{11}-\psi'C_{33}
&=
-\frac{3x}{2\tau},
\label{eq:C11-critical}\\
\psi'(C_{22}-C_{33})
&=
-\frac{3(\tau-x)}{2\tau},
\label{eq:C22-critical}\\
C_{11}+C_{22}+C_{33}
&=T.
\label{eq:C-trace}
\end{align}
Solving \eqref{eq:C11-critical}--\eqref{eq:C-trace} yields
\[
C_{11}
=
\frac{
2\tau\psi'T+3\tau-9x
}{
2\tau(3\psi'+4\tau\psi'')
},
\]
\[
C_{22}
=
\frac{
2\tau^2\psi'\psi''T
-3\tau^2\psi''
+3\tau x\psi''
+\tau(\psi')^2T
-3\tau\psi'
+\frac92x\psi'
}{
\tau\psi'(3\psi'+4\tau\psi'')
},
\]
and
\[
C_{33}
=
\frac{
2\tau\psi'\psi''T
+3\tau\psi''
-3x\psi''
+(\psi')^2T
+\frac32\psi'
}{
\psi'(3\psi'+4\tau\psi'')
}.
\]

Substituting these expressions, together with the formula for $C_{12}$,
into the functional and simplifying gives
\eqref{eq:P-explicit}.

Since $\mathcal P(\tau,x)$ is quadratic in $x$, differentiating
\eqref{eq:P-explicit} twice gives
\[
\frac{\partial^2\mathcal P}{\partial x^2}
=
-\frac{
9(\psi'')^2
}{
\psi'(\psi'+\tau\psi'')
(3\psi'+4\tau\psi'')
},
\]
which proves \eqref{eq:P-xx}.

Next, setting $x=0$ in \eqref{eq:P-explicit} and differentiating twice
with respect to $\tau$ gives
\[
\frac{d^2}{d\tau^2}\mathcal P(\tau,0)
=
-\frac{
4(\psi'')^2(2T\psi'+3)^2
}{
(3\psi'+4\tau\psi'')^3
}.
\]
Similarly, setting $x=\tau$ first and then differentiating twice gives
\[
\frac{d^2}{d\tau^2}\mathcal P(\tau,\tau)
=
-\frac{
16(\psi'')^2(T\psi'-3)^2
}{
(3\psi'+4\tau\psi'')^3
}.
\]
All denominators are positive by the inequalities established above, so
the three second derivatives are nonpositive.

Finally, when $\tau=0$, necessarily $x=0$ and $V=0$. The value
$\mathcal P(0,0)$ is obtained by taking the limit $\tau\downarrow0$ in
\eqref{eq:P-explicit}.
\end{proof}

 We next prove the three vertex inequalities in the proof of Proposition~\ref{prop: no interior first contact}.

\begin{lemma}\label{lem:three vertices ineq}
Let $\lambda\ge16$ and $s\ge8$. Let $\mathcal Q(\tau,x)$ denote the right-hand side of the determinant estimate in \cref{Ineq: L det K / 2det B}. Then
$$
\mathcal Q(0,0)<0,\qquad
\mathcal Q(s,0)<0,\qquad
\mathcal Q(s,s)<0.
$$
Equivalently, the determinant upper bound is strictly negative at each of the three extremal configurations identified in Lemma~\ref{lem: det reduction}.
\end{lemma}
\begin{proof}
Let $\mathcal Q_\lambda(\tau,x)$ denote the upper bound in \cref{Ineq: L det K / 2det B}. We first observe that each of the three vertex values is decreasing in $\lambda$. Indeed, $T=\lambda+s-4\psi$, while the part of $\mathcal Q$ outside the term $-\mathcal P$ has $\lambda$-derivative $-2\psi\psi'<0$. Direct differentiation of the three minimized expressions gives
\begin{equation*}
\begin{gathered}
\frac{\partial}{\partial T} \mathcal{P}(0,0)=\frac{2 T \psi^{\prime}+3}{3}>0, \\
\frac{\partial}{\partial T} \mathcal{P}(s, 0)=\frac{\left(\psi^{\prime}+2 s \psi^{\prime \prime}\right)\left(2 T \psi^{\prime}+3\right)}{3 \psi^{\prime}+4 s \psi^{\prime \prime}}>0,
\end{gathered}
\end{equation*}
and
$$
\frac{\partial}{\partial T}\mathcal P(s,s)
=
\frac{\psi'\bigl(2T(\psi'+2s\psi'')+3\bigr)}
{3\psi'+4s\psi''}>0.
$$

Here we used $\psi'>0$ and $\psi'+2s\psi''>0$. Thus it suffices to prove the inequalities for $\lambda=16$.
Put
$$
r=\sqrt{1+s}\ge3,
\qquad
L=\log r.
$$
Then
$$
\psi=r-\frac1r-2L,
\qquad
\psi'=\frac{(r-1)^2}{2r^3},
\qquad
\psi''=\frac{(r-1)(3-r)}{4r^5}.
$$
We treat the three vertices separately.
\begin{enumerate}
    \item The vertex $(\tau,x)=(0,0)$.
The inequality
$
L\ge1-\frac1r
$
gives
$
0<\psi\le \frac{(r-1)^2}{r},
$
and hence
$$
\psi^2+4\psi<s.
$$

Moreover,
$
2\psi'+s\psi''>0.
$
Consequently, the part of $\mathcal Q_{16}(0,0)$ preceding $-\mathcal P(0,0)$ is strictly negative.

It remains to note that $\mathcal P(0,0)>0$. Since
$$
\mathcal P(0,0)
=
\frac{2T^2(\psi')^2+6T\psi'-9}{6\psi'},
$$
it is enough to show $T\psi'>3/2$. Using $L\ge1$,
$$
T
=
r^2-4r+15+\frac4r+8L
\ge
r^2-4r+23+\frac4r.
$$
Writing $r=3+\xi$, one obtains
$$
2r^4\left(T\psi'-\frac32\right)
\ge
\xi^5+6\xi^4+14\xi^3+26\xi^2+36\xi+13>0.
$$
Therefore
$
\mathcal Q_{16}(0,0)<0.
$

\item The vertex $(\tau,x)=(s,0)$.
Since
$$
\psi'+\sqrt{(\psi')^2+\psi\psi''}\ge\psi',
$$

replacing the square-root coefficient by $\psi'$ gives a larger upper bound. After substituting the formulas above, this upper bound of $\mathcal Q_{16}(s,0)$ is a convex quadratic polynomial in $L$.

Set
$
y=\frac{r-1}{r+1}.
$
Since $L=2\operatorname{arctanh}y$,
$$
2y
\le L\le
2y+\frac{2y^3}{3(1-y^2)}.
$$

By convexity, it therefore suffices to evaluate the upper bound at these two endpoints $2y$ and $2y+\frac{2y^3}{3(1-y^2)}$. After clearing the positive denominators and writing $r=3+\xi$, the negatives of the two resulting numerators are polynomials in $\xi$ with strictly positive coefficients. Hence
$
\mathcal Q_{16}(s,0)<0.
$

\item The vertex $(\tau,x)=(s,s)$.
Set
\begin{equation}\label{eq:theta def}
    \theta
:=
\frac{\sqrt{(\psi')^2+\psi\psi''}}{\psi'}.
\end{equation}
Since $\psi\ge0$ and $\psi''\le0$ for $r\ge3$, we have $\theta\le1$.
Moreover, using
\[
\psi=r-\frac1r-2L,
\qquad
\psi'=\frac{(r-1)^2}{2r^3},
\qquad
\psi''=\frac{(r-1)(3-r)}{4r^5},
\]
we compute
\[
\theta^2
=
\frac{2r(r-3)L+4(r-1)}{(r-1)^3}.
\]
Since
\[
L=\log r
=
2\operatorname{arctanh}\frac{r-1}{r+1}
\ge
\frac{2(r-1)}{r+1},
\]
it follows that
\[
\theta^2
\ge
\frac4{r+1}
\ge
\frac{64}{(r+5)^2},
\]
where the last inequality is equivalent to $(r-3)^2\ge0$. Hence
\[
\frac8{r+5}\le\theta\le1.
\]
Therefore
\[
(1-\theta)
\left(
\theta-\frac8{r+5}
\right)\ge0,
\]
which gives
\[
\theta
\ge
\frac{(r+5)\theta^2+8}{r+13}.
\]
Multiplying by $\psi'$ and using \cref{eq:theta def}
we obtain
\begin{equation}
\label{eq:strong-square-root-bound}
\psi'
+\sqrt{(\psi')^2+\psi\psi''}
\ge
2\psi'
+
\frac{r+5}{r+13}\frac{\psi\psi''}{\psi'}.
\end{equation}

We next derive a rational upper bound for $\psi$. Set $$f(r)= \log r-3+\frac{10}{r+2} \quad(r\geq 3).$$ We compute \[f^{\prime}(r)=\frac{1}{r}-\frac{10}{(r+2)^2}=\frac{r^2-6 r+4}{r(r+2)^2}\] Hence $f$ has a global minimum at $r_0=3+\sqrt{5}.$ Numerically, $f_{\min}\approx 0.038>0.$
Thus
\[
L\ge3-\frac{10}{r+2},
\]
and consequently
\begin{equation}
\label{eq:psi-upper-hopf}
0\le\psi
=
r-\frac1r-2L
\le
\frac{r^3-4r^2+7r-2}{r(r+2)}
=:z_*(r).
\end{equation}

We now apply \eqref{eq:strong-square-root-bound} to the vertex $(s,s)$.
For $z\ge0$, set
\[
\mathcal T_z:=16+s-4z
\]
and define
\[
\mathcal K_z(s,s)
:=
\frac{\psi'\mathcal T_z^2}{2}
-
\frac{(\psi'\mathcal T_z-3)^2}
{2(3\psi'+4s\psi'')}.
\]
Here $\psi'$ and $\psi''$ are regarded as fixed functions of $r$. Define
\begin{align*}
\mathcal F(r,z)
:={}&
z^2+4z
-\psi'\bigl(4z^2+2z\mathcal T_z+6s\bigr)
-2\psi''z^2s+3s\\
&\quad
-8s(2+z\psi')
\left(
2\psi'
+
\frac{r+5}{r+13}\frac{z\psi''}{\psi'}
\right)
-\mathcal K_z(s,s).
\end{align*}
Since the coefficient of
\[
\psi'+\sqrt{(\psi')^2+\psi\psi''}
\]
in the expression for $\mathcal Q_{16}(s,s)$ is nonpositive,
\eqref{eq:strong-square-root-bound} gives
\[
\mathcal Q_{16}(s,s)
\le
\mathcal F(r,\psi).
\]

As a function of $z$, $\mathcal F(r,z)$ is strictly convex. Indeed,
direct differentiation gives
\[
\frac12\frac{\partial^2\mathcal F}{\partial z^2}
=
1-4\psi'
+
\frac{8(\psi')^2}{3\psi'+4s\psi''}
-
\left(
2+\frac{8(r+5)}{r+13}
\right)s\psi''.
\]
By Eq.~(3.6),
\[
3\psi'+4s\psi''
=
\psi'+2(\psi'+2s\psi'')>0,
\]
while $\psi''\le0$. Moreover,
\[
0<\psi'
=
\frac{(r-1)^2}{2r^3}
\le
\frac2{27}
\qquad (r\ge3).
\]
Therefore
\[
\frac12\frac{\partial^2\mathcal F}{\partial z^2}
\ge
1-4\psi'
\ge
\frac{19}{27}>0.
\]
Therefore $\mathcal F(r,\cdot)$ is convex in its second variable.
By \eqref{eq:psi-upper-hopf}, $0\le\psi\le z_*(r).$
Since a convex function on an interval attains its maximum at an endpoint,
\[
\mathcal Q_{16}(s,s)
\le
\mathcal F(r,\psi)
\le
\max
\left\{
\mathcal F(r,0),
\mathcal F\bigl(r,z_*(r)\bigr)
\right\}.
\]

It remains to check these two endpoint values. First,
\[
\mathcal F(r,0)
=
-\frac{\mathcal R_0(r)}
{2r^3(r-1)^2(r^2+4r+6)},
\]
where
\[
\begin{aligned}
\mathcal R_0(r)
={}&
r^6\left[
2(r-3)^3+9(r-3)^2+168(r-3)+27
\right]\\
&+
r^3\left(336r^2+80r-872\right)
+(1770r^2-1490r+447).
\end{aligned}
\]
All three brackets are positive for $r\ge3$. Hence $\mathcal F(r,0)<0.$

For the other endpoint,
\[
\mathcal F\bigl(r,z_*(r)\bigr)
=
-\frac{\mathcal R_1(r)}
{2r^7(r-1)^2(r+2)^2(r+13)(r^2+4r+6)},
\]
where
\[
\begin{aligned}
\mathcal R_1(r)
={}&
r^{12}\Bigl[
2(r-3)(r+13)(r-13)^2
+82(r-8)^2+43r+611
\Bigr]\\
&+
r^9\left(16134r^2-52650r+12874\right)\\
&+
r^6\left(93776r^2-13122r-191708\right)\\
&+
r^3\left(148835r^2+103537r-225518\right)\\
&+
147310r^2-43216r+4872.
\end{aligned}
\]
The first bracket is manifestly positive. Each of the remaining
quadratic polynomials is increasing on $[3,\infty)$ and is positive at
$r=3$; for example,
\[
16134r^2-52650r+12874
=
16134(r-3)^2+44154(r-3)+130>0.
\]
Thus
\[
\mathcal R_1(r)>0
\qquad (r\ge3),
\]
and hence
\[
\mathcal F\bigl(r,z_*(r)\bigr)<0.
\]

We conclude that
\[
\mathcal Q_{16}(s,s)<0
\qquad (r\ge3).
\]

\end{enumerate}
The three vertex inequalities follow.
\end{proof}

\printbibliography
\end{document}